\documentclass[]{amsart}

\usepackage[utf8]{inputenc}
\usepackage[OT2,T1]{fontenc}
\DeclareSymbolFont{cyrletters}{OT2}{wncyr}{m}{n}
\DeclareMathSymbol{\Sha}{\mathalpha}{cyrletters}{"58}

\title[Density of Selmer ranks in families of even Galois representations]{
Density of Selmer ranks in families of even Galois representations, Wiles' formula, 
and
global reciprocity
}
\date{\today}
\usepackage[foot]{amsaddr}
\author{Peter Vang Uttenthal}
\email{pu42@cornell.edu}
\address{Department of Mathematics, Cornell University, Ithaca, NY 14853-4201, USA}
\newcommand{\Gal}{\mathrm{Gal}}
\newcommand{\cl}{\mathrm{Cl}}
\newcommand{\Q}{\mathbb{Q}}
\newcommand{\Z}{\mathbb{Z}}
\newcommand{\F}{\mathbb{F}}
\newcommand{\W}{\mathbb{W}(\mathbb{F}_4)}
\newcommand{\Wp}{\mathbb{W}(\mathbb{F}_{p^f})}
\newcommand{\p}{\mathfrak{p}}

\usepackage{amsthm,amsmath, mathrsfs}
\usepackage{amsfonts}
\usepackage{amssymb}
\usepackage{fancyhdr}
\usepackage{IEEEtrantools}
\usepackage{tikz-cd}
\usepackage[english]{babel}
\usepackage[utf8]{inputenc}

\usepackage{csquotes}

\newtheorem{theorem}{Theorem}
\newtheorem{lemma}[theorem]{Lemma}
\newtheorem{definition}[theorem]{Definition}
\newtheorem{proposition}[theorem]{Proposition}
\newtheorem{corollary}[theorem]{Corollary}

\usepackage[hyphens,spaces,obeyspaces]{url}
\usepackage[colorlinks,allcolors=blue,hyperindex,breaklinks]{hyperref}

\hypersetup{
           breaklinks=true,   
           colorlinks=true,   
           pdfusetitle=true,  
        }

\begin{document}

\begin{abstract}
This paper concerns the distribution of Selmer ranks in a family of even Galois representations in residual characteristic $p=2$ obtained by allowing ramification at auxiliary primes. 
The main result is a Galois cohomological analogue of a theorem of Friedlander, Iwaniec, Mazur and Rubin on the distribution of Selmer ranks in a family of twists of elliptic curves. 
The Selmer groups are constructed as prescribed by the Galois cohomological 
method for $\operatorname{GL}(2)$: 
At each ramified place, the local Selmer condition 
is the tangent space of a smooth quotient of the local deformation ring. 
By methods of global class field theory, the Selmer group at the minimal level is computed explicitly. 
The infinitude of primes for which the Selmer rank increases by one is proved, and the density of such primes is shown to be 1/192. 
The proof combines Wiles' formula and the global reciprocity law. 
The result has implications for the algebraic structure of
even deformation rings and the distribution of their presentations in families.  
\end{abstract}

\maketitle
\tableofcontents

\section{Introduction}
In \cite{FIMR}, Friedlander, Iwaniec, Mazur and Rubin introduce the notion of spin of prime ideals and apply it to the arithmetic statistics of elliptic curves. In particular, they prove a theorem on the distribution over primes of the Selmer rank in a family of twists of elliptic curves. In this paper, we find a similar distribution over primes of the Selmer rank in a family of even Galois representations in residual characteristic $p=2$. 

Let $E$ be the elliptic curve over $\Q$ defined by 
$$
E: y^2 = x^3 + x^2 - 16x - 29.
$$
Let $K$ be the totally real subfield of $\Q(\zeta_7)$ of degree 3 over $\Q$.
Then $K = \Q(E[2])$, the field generated by the 2-torsion points of $E$ \cite{FIMR}. For a prime $p$, the quadratic twist $E^{(p)}$ of $E$ by $p$ is the elliptic curve
$$
E^{(p)}: py^2 = x^3 + x^2 - 16x - 29.
$$
Let $\mathfrak{p}$ be a prime of $K$ with totally positive generator $\pi$, and let $g$ be a generator of 
$\Gal(K/\mathbb{Q})$. Define
$$
\text{spin}(\mathfrak{p}) = \bigg(  \frac{\pi}{\mathfrak{p}^{g}} \bigg)
$$
where $(\pi/\mathfrak{a})$ is the quadratic residue symbol in $K$.

\begin{theorem}[Friedlander, Iwaniec, Mazur, Rubin] \label{FIMRthm}
Let $p$ be a prime that splits completely in $K$, and let $\mathfrak{p}$ be a prime above $p$ in $K$. Suppose $\mathfrak{p}$ splits completely in $\Q(E[4])$. Then
$$
\dim_{\F_2} \text{\upshape{Sel}}_2(E^{(p)}/\Q) =
\begin{cases}
\dim_{\F_2} \text{\upshape{Sel}}_2(E/\Q) + 2 & \text{if } \text{\upshape{spin}}(\mathfrak{p}) = 1, \\
\dim_{\F_2} \text{\upshape{Sel}}_2(E/\Q)  & \text{if } \text{\upshape{spin}}(\mathfrak{p}) = -1. \\
\end{cases}
$$ 
There are infinitely many primes $p$ for which the Selmer rank increases by two. In fact, this happens half the time. 
\end{theorem}

In this paper, we study an even, reducible and completely decomposable residual Galois representation
$$
\overline{\rho}: \Gal(\overline{\Q}/\Q) \to \text{GL}(2, \F_4)
$$
unramified outside a finite set of places $S$. 
The field 
$\Q(\overline{\rho})$ fixed by $\ker(\overline{\rho})$ equals the totally real field $\Q(E[2])$ of the elliptic curve $E: y^2 = x^3 + x^2 - 16x - 29$. 
The set $S$ contains the archimedian place, $\infty$, the characteristic $p=2$ of the field of coefficients of $\overline{\rho},$ and the ramified prime $7$ of $K$.
We lift the representation \emph{irreducibly} to mod 8 by thickening the image with cohomology classes ramified at three auxiliary primes $l$, $v$, and $w$, 
which are consequently appended to the set of places $S$ at which we allow ramification.
The primes $l,v$ and $w$ are defined by Chebotarev conditions, and the smallest examples with the desired properties are $l=17, v=167$ and $w=379$.
For a place $q$ in $\Q$, let $G_q=\Gal(\overline{\Q}_q/\Q_q)$. 
For each place $q$ in the set $S=\{ \infty, 2, 7, l, v, w\}$,
we find smooth quotients of the local deformation rings with corresponding tangent spaces 
$\mathcal{N}_q \subseteq H^1(G_q, Ad(\overline{\rho}))$. 
Let $G_S$ denote the Galois group of the maximal extension of $\Q$ unramified outside $S$, and let 
$Ad(\overline{\rho})$ be the set of 2-by-2 matrices over $\F_4$ equipped with a Galois action via $\overline{\rho}$ and conjugation. 
The Selmer group $H^1_{\mathcal{N}}(G_S, Ad(\overline{\rho}))$ is defined to be the kernel of the map
$$
H^1(G_S, Ad(\overline{\rho})) \to \bigoplus_{q\in S} H^1(G_q, Ad(\overline{\rho})) / \mathcal{N}_q.
$$ If $M$ is a Galois module, we let $M^* = \operatorname{Hom}(M, \F_2)$ denote the dual module. 
For each $q\in S$, let $\mathcal{N}_q^\perp \subseteq H^1(G_q,Ad(\overline{\rho})^*)$ denote the annihilator of $\mathcal{N}_q$ under the local pairing. The dual Selmer group $H^1_{\mathcal{N}^\perp}(G_S, Ad(\overline{\rho})^*)$ is the kernel of the map 
$$
H^1(G_S, Ad(\overline{\rho})^*) \to \bigoplus_{q\in S} H^1(G_q, Ad(\overline{\rho})^*) / \mathcal{N}_q^\perp.
$$
Let $K^{(w)}$ denote the maximal 2-elementary abelian extension of $K$ unramified outside the primes in $K$ above $w$. Let $f^{(w)} \in H^1(G_S, Ad(\overline{\rho}))$ be the cohomology class corresponding to $K^{(w)}$, and let $\sigma_p$ be a lift of Frobenius at $p$ in $G_p$. There is an extension of $K$ 
which, in our setting, controls the existence of 2-elementary abelian extensions of $K$ with prescribed tame ramification,
called the governing field of $K$. More details about this notion can be found in \cite{gras}.
The main result of this paper is the following. 
\begin{theorem} \label{maintheorem}
Let $p$ be a prime that splits completely in $K$ and let $\mathfrak{p}$ be a prime above $p$ in $K$. Suppose $\mathfrak{p}$ satisfies a splitting condition in the governing field of $K$.
Then 
$$ \dim H^1_\mathcal{N}(G_{S\cup \{ p \}}, Ad(\overline{\rho}))
=
\begin{cases}
\dim H^1_\mathcal{N}(G_{S}, Ad(\overline{\rho})) + 1, & \text{if }f^{(w)}(\sigma_p) = 0 , \\
\dim H^1_\mathcal{N}(G_{S}, Ad(\overline{\rho})), &  \text{if }f^{(w)}(\sigma_p) \neq 0. \\
\end{cases}
$$
There are infinitely many primes $p$ for which the Selmer rank increases by one. In fact, this happens a quarter of the time. 
\end{theorem}
We choose $\mathcal{N}_\infty = 0,$ which ensures that 
all global deformations of $\overline{\rho}$ that locally factor through the smooth quotients at the places in $S$ are \emph{even}. 
On the other hand, Galois representations attached to elliptic curves are \emph{odd}, and hence our result is distinct from Friedlander, Iwaniec, Mazur and Rubin's Theorem \ref{FIMRthm}.

Since the pair $(H^1_\mathcal{N}(G_{S}, Ad(\overline{\rho})), H^1_{\mathcal{N}^\perp}(G_{S}, Ad(\overline{\rho})^*) )$ of the global setting determines the algebraic structure of 
the global, even deformation ring  $R_{(\mathcal{N}_q)_{q\in S}}$
parametrizing deformations of 
$\overline{\rho}$ that locally factor through the smooth quotients for all $q\in S$, the following result on even deformation rings is a corollary of Theorem \ref{maintheorem}.
\begin{theorem}
    The global, even deformation ring 
    $R_{(\mathcal{N}_q)_{q\in S}}$
    parametrizing deformations of 
    $\overline{\rho}$ that locally factor through the smooth quotients for all $q\in S$
   has the structure 
$$
R_{(\mathcal{N}_q)_{q\in S}} \simeq \W[[T_1,T_2]]/(r_1,r_2, r_3).
$$
There exists an infinite set of primes $p$ of density 1/192 such that the even deformation ring
$R_{(\mathcal{N}_q)_{q\in S\cup\{p\}}}$ 
parametrizing deformations
that locally factor through the smooth quotients for all $q\in S\cup\{p\}$
has the structure 
$$
R_{(\mathcal{N}_q)_{q\in S\cup\{p\}}}
 \simeq \W[[T_1,T_2, T_3]]/(r_1,r_2, r_3, r_4).
$$
\end{theorem}

\subsection{Related work}
The distribution results on spin in \cite{FIMR} are conditional on
a conjecture on short character sums,  denoted $C_n$, where $n$ is the degree of the number field over which the spin symbol is defined. The conjecture $C_n$ is only proved for $n=3$.
Conditionally on $C_n$ for $n\geq 3$, Koymans and Milovic (\cite{Koymans}) have proved that the spin symbol does not define a Chebotarev condition. 
We thank Peter Koymans for explaining the proof to us, which follows from the equidistribution result on joint spin symbols in \cite{Koymans}.
In our case, the infinitude of primes for which the Selmer rank increases follows from Chebotarev's theorem applied to abelian extensions of $K$.  
The techniques in \cite{FIMR} are mainly drawn from analytic number theory. In this paper, the main technique is the the purely Galois cohomological method of Ramakrishna and others (\cite{even1}, \cite{even2}, \cite{lifting}, \cite{SFM}, \cite{KR} \cite{artinII}). This method extends the work of Mazur (\cite{Mazur1}, \cite{Mazur2}) and is built upon
Galois cohomology and class field theory as essential tools. In addition, we use the tame Gras-Munnier theorem to construct certain global cohomology classes corresponding to number fields with a prescribed tame ramification.
See \cite{gras} for a new proof of the tame Gras-Munnier theorem. 

Two-dimensional Galois representations of the absolute Galois group over $\Q$ are \emph{odd} if they cut out number fields that are ramified at infinity, while they are \emph{even} if they cut out totally real fields. The odd case is the most well understood, because odd Galois representations can be studied with tools from algebraic geometry, such as Jacobians of modular curves as in the work of Ribet
(\cite{Ribet}) and \'{e}tale cohomology of algebraic varieties.  The work of Wiles and others suggest that deformation rings of odd Galois representations which are unramified outside a finite set of primes, have fixed determinant, and are potentially semistable, are always finite and flat over $\Z_p$ (see e.g. \cite{canadian}, p. 208). On the other hand, it is not known in general whether deformation rings of even Galois representations are finite and flat over $\Z_p$.
Even representations in characteristic zero are conjectured to come from geometry in the sense of Fontaine-Mazur, but they are also conjectured to be Artin representations, so the geometry is rather trivial. Meanwhile,  
complex even representations conjecturally correspond to automorphic representations attached to classical Maass wave forms. In a recent survey of the progress on Langlands reciprocity since Wiles' proof of Fermat's Last Theorem, Calegari (\cite{Calegari}, p. 28) describes our ignorance of even representations, and how limited our progress on understanding the role of even representations in the Langlands program has been so far. 
The aim of this paper is to shed light on the extent to which even representations behave like the odd case. The techniques we use often work well without any oddness assumption, and this feature is exploited to study an explicit family of even representations. 
\\
\\
In \cite{even2}, the study of the arithmetic statistics of 
primes raising the level of an even, 3-adic, surjective Galois representation was initiated, and level raising at a prime $p$ was reduced to a splitting condition in an extension that depended on the prime $p$ itself.
Since this condition is most likely not Chebotarev, proving the infinitude of such primes seems out of reach with current techniques. 
The work by \cite{FIMR} on spin of prime ideals takes a step forward in the direction of proving infinitude of primes satisfying splitting conditions in extensions that depend on the prime. Further progress in this direction has been achieved by McMeekin in \cite{Christine1} and by Chan, McMeekin and Milovic  in \cite{Christine2}. However, the results involving spin only apply unconditionally for cubic, totally real number fields and for splitting conditions in 2-extensions of these, and therefore they cannot be applied to the 3-adic representation of \cite{even2}.
For this to work, one would have to generalize the results in \cite{FIMR} to $A_4$-extensions of $\Q$ and work with a \emph{cubic} residue symbol. A promising step forward have been taken by Koymans and Milovic (\cite{Koymans}) who, conditionally on $C_n$, have extended the results of \cite{FIMR} to all Galois extensions. In the future, it would be of interest to discover to what extent their results can be generalized to cubic residue symbols. One of the original motivations for this paper was to construct a deformation-theoretic setting where one could potentially apply the techniques in \cite{FIMR} to prove infinitude of primes for which the Selmer rank increased. This is why we build our residual representation out of the cyclic, totally real field $K$. In addition, working in characteristic $p=2$ is what is needed for lifts of $\overline{\rho}$ to cut out 2-extensions of $K$. However, in the end we were able to prove infinitude of primes raising the Selmer rank without using prime spin. 

To work with a residual representation that cuts out the field $K$, we were naturally led to study a \emph{reducible} $\overline{\rho}$. For this reason, we were able to draw on ideas of Hamblen and Ramakrishna (\cite{HR}) and 
Fakhruddin, Khare and Patrikis (\cite{FKP}), who work with \emph{trivial} primes in residual characteristic $p\geq 3.$
This notion turns out to be the right one for us to consider in this paper as well, suitably adapted to characteristic $p=2$.

It is generally not obvious how to increase the Selmer rank by allowing ramification at just one new prime. 
One of the original motivations of Khare, Larsen, and Ramakrishna in \cite{KLR1} for developing the so-called \emph{doubling method} (see \cite{FKP1}, p. 3516) was to be able to raise the Selmer rank by allowing ramification at \emph{two} new primes simultaneously. 
The precise result they prove is stated below. For a precise definition of \emph{nice} primes, see \cite{KLR1}, Definition 1, p. 713. 

\begin{theorem}[Khare, Larsen, and Ramakrishna (\cite{KLR1} Corollary 14)] \label{KLRselmer}
Let $p\geq 5$, and let $\overline{\rho}$ be odd, absolutely irreducible, and modular of square-free level and weight 2. Let $S$ be a set of places containing $p$ and the ramified primes of $\overline{\rho}$. There is a finite set $T$ of nice primes such that 
$$
H^1_\mathcal{N}(G_{S\cup T}, Ad(\overline{\rho})) =0.
$$ 
Furthermore, there exists a set $U$ consisting of at most two nice primes such that 
$$
\dim H^1_\mathcal{N}(G_{S\cup T \cup U}, Ad(\overline{\rho})) \geq 1.
$$
\end{theorem}

In the setting of \cite{KLR1}, increasing the Selmer rank by allowing ramification at two additional primes has applications to the existence of newforms, as the following corollary shows. 

\begin{corollary}[\cite{KLR1}, p. 727]
In the notation of the previous theorem, there is a unique newform of level $S\cup T$ whose $p$-adic Galois representation is congruent to $\overline{\rho}$ mod $p$, while there exist at least two possibly Galois conjugate newforms of level $S\cup T \cup U$ whose residual representations are isomorphic to $\overline{\rho}$.
\end{corollary}

In the explicit setup of this paper, we are able to raise the Selmer rank by allowing ramification at just a single new prime. 

\subsection{Summary}

Finally, we summarize the contents of the following sections. 

We construct the Selmer condition $\mathcal{N}=(\mathcal{N}_q)_{q\in S}$
as prescribed in the Galois cohomological method (\cite{SFM}): For each place $q \in S$,
we find a pair
$(\mathcal{C}_q,\mathcal{N}_q)$, where $\mathcal{C}_q$ is a class of 
liftable, local
deformations that factor through a smooth quotient of the local deformation ring at $q$, and $\mathcal{N}_q$ is the space of local cohomology classes that preserves $\mathcal{C}_q$. 
The Selmer condition $\mathcal{N}$ defines a global, versal deformation ring
$R_{(\mathcal{N}_q)_{q\in S}}$ which is a hull for the deformation functor
that takes coefficient rings to $\overline{\rho}$-deformations $\rho$
with the property that $\rho|_{G_q} \in \mathcal{C}_q$ for all $q\in S$.
The beginning sections are devoted to ensuring that the 
deformations that factor through $R_{(\mathcal{N}_q)_{q\in S}} $
are irreducible. This implies that such representations are not just sums of characters described by class field theory. 
We use an explicit technique of thickening the image of the Teichm\"{u}ller lift by acting on it with cohomology classes that cut out various irreducible Galois modules.

The data
$(
H^1_{\mathcal{N}}(G_S, Ad(\overline{\rho}))
,H^1_{\mathcal{N}^\perp}(G_S, Ad(\overline{\rho})^*))$
of the global setting describe the algebraic structure of the even deformation ring 
$R_{(\mathcal{N}_q)_{q\in S}}$, in the sense that if
$$
(\dim
H^1_{\mathcal{N}}(G_S, Ad(\overline{\rho}))
,\dim H^1_{\mathcal{N}^\perp}(G_S, Ad(\overline{\rho})^*))
=(n,m)
$$
then 
$$
R_{(\mathcal{N}_q)_{q\in S}} \simeq \W[[T_1, \ldots, T_n]]/(r_1,\ldots,r_m).
$$
We compute the Selmer group
$H^1_{\mathcal{N}}(G_S, Ad(\overline{\rho})$ explicitly, 
and apply Wiles' formula (\ref{Wiles}) to show to that
$$
(\dim
H^1_{\mathcal{N}}(G_S, Ad(\overline{\rho}))
,\dim H^1_{\mathcal{N}^\perp}(G_S, Ad(\overline{\rho})^*))
=
(2,3).
$$
This is a consequence of a formula from global class field theory concerning the rank of certain ray class groups. 

We do not study the existence of $\W$-valued points of $R_{(\mathcal{N}_q)_{q\in S}}$. 
Instead, we focus solely on studying the variation in the rank of the global setting as we allow ramification at new trivial primes $p$.
If $p$ splits completely in a field we denote $K^{(w)}$, then the Selmer rank increases by one, and otherwise the Selmer rank remains the same. 
The proof combines Wiles' formula and the global reciprocity law. In particular, Wiles' formula governs
what additional global cohomology classes that arise after allowing ramification at $p$, while the global reciprocity law provides control of the local shape of these new cohomology classes at the places in $S$.
We show that it is possible to lower the rank of the Selmer group by changing bases locally at trivial primes, but that the rank is bounded below by one. 
In the final section, we prove that the density of primes raising the Selmer rank is 1/192, and we deduce a corollary concerning the statistics of global, even deformation rings. 

\subsection{Acknowledgements}
I am thankful to Ravi Ramakrishna for getting me interested in the project, for his kindness in sharing his ideas, and for his support and encouragement during the preparation of this paper. I would like to thank Peter Koymans for helpful conversations on spin of primes. I am grateful to the anonymous referee for a careful reading of the manuscript, including many helpful suggestions that improved the final version of the paper.

\section{Review of some aspects of Galois deformations}
The purpose of this section is to briefly summarize some standard facts from deformation theory that will be used in the paper.
Let $p$ be a prime and let $G_\Q=\Gal(\overline{\Q}/\Q).$
The theory is available for Galois representations 
$$
\overline{\rho}: G_\Q \longrightarrow G'(\overline{\F}_{p})
$$
valued in arbitrary reductive algebraic groups $G'$ (\cite{Patrikis1}, \cite{FKP1}),
and is concerned with lifting $\overline{\rho}$
to coefficient rings in the category of complete Noetherian local rings. Below, we will specialize to $\operatorname{GL}(2)$, although more general statements are available. 
Let $G'=\operatorname{GL}(2),$ and let 
$f \geq 1$ be such that 
$$
\overline{\rho}: G_\Q \longrightarrow \operatorname{GL}(2,\F_{p^f}).
$$
Suppose $S$ is a finite set of places containing $p$.
If one has applications of the global duality theorems of Galois cohomology in mind, or theorems relying on these, then it is convenient to have $S$ contain the archimedian place $\infty$.
If $\overline{\rho}$ is unramified outside $S$, it factors through 
the Galois group $G_S =\Gal(\Q_S/\Q)$, where 
$\Q_S$ is the maximal extension of $\Q$ unramified outside $S$.

\begin{proposition}[Existence of versal deformation rings]
Let $G$ be the absolute Galois group of a local field or a number field and consider a continuous homomorhpism 
$$
\overline{\rho}: G \to \operatorname{GL}(2,\F_{p^f}).
$$
The deformation functor of $\overline{\rho}$ 
from
the category of complete Noetherian local $\Wp$-algebras
to the set of strict equivalence classes of deformations of $\overline{\rho}$
has a hull $(R^{\operatorname{ver}}, \rho^{\operatorname{ver}}).$ 
In particular, for any deformation 
$$
\rho: G \longrightarrow \operatorname{GL}(2,A)
$$
of $\overline{\rho}$, there exists a morphism 
$$
\varphi: R^{\operatorname{ver}} \to A
$$
such that 
$$
\varphi \circ \rho^{\operatorname{ver}} = \rho.
$$
The morphism need not be unique. 
\end{proposition}

\begin{proposition}[From local to global liftability] \label{local-global lift}
Consider a global residual representation
$$
\overline{\rho}: G_S \to \operatorname{GL}(2,\F_{p^f}),
$$
unramified outside a finite set of places $S$, 
and a deformation 
$$
\rho^{(k)}: G_S \to \operatorname{GL}(2,\Wp/ p^k \Wp)
$$
for some $k \geq 1.$
Let 
$\Sha_S^2(Ad(\overline{\rho}))$
be the kernel of the map 
$$
H^2(G_S, Ad(\overline{\rho})) \to \bigoplus_{v\in S} 
H^2(G_v, Ad(\overline{\rho}) ).
$$
Suppose 
$
\rho^{(k)}|_{G_v}
$
is liftable to $\Wp/ p^{k+1} \Wp$ for all $v\in S.$
Then the global obstruction to lifting $\rho^{(k)}$ to $\Wp/ p^{k+1} \Wp$
lives in $\Sha_S^2(Ad(\overline{\rho}))$. In particular, if $\Sha_S^2(Ad(\overline{\rho}))$ is trivial, then there exists a global lift to $\Wp/ p^{k+1} \Wp$.
\end{proposition}

\begin{proposition}[The action of cohomology classes on lifts]
Let $S$ be a finite set of places of $\Q$.
Let $G$ be $G_S$, $G_\Q$, or $G_v$ for some place $v$ in $\Q$. 
    Consider a representation
$$
\rho^{(k-1)}: G \to \operatorname{GL}(2,\Wp/ p^{k-1} \Wp)
$$
for some $k \geq 2.$
The group
$H^1(G, Ad(\overline{\rho}))$ acts transitively on infinitesimal deformations of 
$\rho^{(k-1)}$. More precisely, 
if 
$$
\rho^{(k)}: G \to \operatorname{GL}(2,\Wp/ p^k \Wp)
$$
is a deformation of $\rho^{(k-1)}$ and 
$f \in H^1(G, Ad(\overline{\rho})),$
then 
$$
(I+p^{k-1}f) \rho^{(k)}
$$ 
is another deformation of $\rho^{(k-1)}$ ($I$ is the identity matrix), and if 
$\rho^{(k)}_1$ and $\rho^{(k)}_2$ are deformations of $\rho^{(k-1)}$, then there exists some 
$f\in H^1(G, Ad(\overline{\rho}))$
such that 
$$
\rho^{(k)}_2 = (I+p^{k-1}f) \rho^{(k)}_1.
$$ 
\end{proposition}
\begin{proof}
    If $f$ and $f'$ are 1-cocyles that represent the same cohomology class in $H^1(G, Ad(\overline{\rho}))$, then the representations
    $(I+p^{k-1}f) \rho^{(k)}$ and $(I+p^{k-1}f') \rho^{(k)}$ are strictly equivalent. 
    Suppose $\rho^{(k)}_1$ and $\rho^{(k)}_2$ are deformations of $\rho^{(k-1)}$. Then the function
    $
\rho^{(k)}_2 ( \rho_1^{(k)})^{-1}
    $
    on $G$ takes values in $I + p^{k-1} Ad(\overline{\rho})$
    and represents a cohomology class in 
    $H^1(G, I + p^{k-1}Ad(\overline{\rho})).$
    Since 
    $
 I + p^{k-1}Ad(\overline{\rho}) \simeq Ad(\overline{\rho})
 $ as $G$-modules, we conclude that the action is transitive.
\end{proof}

\begin{lemma}[Characterization of smooth quotients] \label{smoothquotients}
Let $v$ be a place of $\Q$. 
Let $\mathcal{C}_v$ be a class of $\overline{\rho}$-deformations with coefficients in the ring $\Wp/p^k\Wp$ for some $k\geq 1,$
and let $\mathcal{N}_v$ be the subspace of $H^1(G_v, Ad(\overline{\rho}))$ that preserves $\mathcal{C}_v$. Suppose for any $\rho^{(k)} \in \mathcal{C}_v$, 
where
$$
\rho^{(k)}: G_v \to \operatorname{GL}(2, \Wp/p^k\Wp),
$$
that there exists some 
$$
\rho^{(k+1)}: G_v \longrightarrow \operatorname{GL}(2, \Wp/p^{k+1}\Wp)
$$
such that
$\rho^{(k)} \in \mathcal{C}_v$ and $\rho^{(k+1)} \equiv \rho^{(k)} \mod p^k.$
Then the local versal deformation ring $R_v$ has a smooth quotient which is a power series ring
in $\dim \mathcal{N}_v$ variables over $\Wp$ with tangent space $\mathcal{N}_v$. 
More precisely, there is a surjective morphism
$$
R_v \to \Wp[[T_1,\ldots, T_{n_v}]]
$$
in the category of complete Noetherian local $\Wp$-algebras such that 
$$
n_v = \dim_{\F_{p^f}} \mathcal{N}_v
$$
and 
such that the induced dual linear map
$$
(\mathfrak{m}_{R_v}/(p, \mathfrak{m}^2_{R_v}))^* \longleftarrow (\mathfrak{m}_{\Wp[[T_1,\ldots, T_{n_v}]]}/(p, \mathfrak{m}^2_{\Wp[[T_1,\ldots, T_{n_v}]]}))^*
$$
coincides with the inclusion 
$$
H^1(G_v, Ad(\overline{\rho})) \longleftarrow \mathcal{N}_v.
$$
Conversely, suppose
there is a surjective morphism
$$
R_v \to \Wp[[T_1,\ldots, T_{n_v}]]
$$
in the category of complete Noetherian local $\Wp$-algebras. Define
$\mathcal{C}_v$ as the class of $\overline{\rho}$-deformations with coefficients in $\Wp/p^k\Wp$ for some $k \geq 1$ that factor through the smooth quotient $\Wp[[T_1,\ldots, T_{n_v}]]$, and
let $\mathcal{N}_v$ be the subspace of $H^1(G_v, Ad(\overline{\rho}))$ obtained by dualizing the induced surjective linear map
$$
\mathfrak{m}_{R_v}/(p, \mathfrak{m}^2_{R_v}) \longrightarrow \mathfrak{m}_{\Wp[[T_1,\ldots, T_{n_v}]]}/(p, \mathfrak{m}^2_{\Wp[[T_1,\ldots, T_{n_v}]]}).
$$
The the deformations in $\mathcal{C}_v$ are always liftable
from $\Wp/p^k\Wp$ to $\Wp/p^{k+1}\Wp$ (and hence to $\Wp$),
and $\mathcal{N}_v$ preserves $\mathcal{C}_v$.
\end{lemma}

Next, we recall Wiles' formula (see Proposition 1.6 in \cite{Wiles}). 

\begin{proposition}[Wiles' formula] \label{Wiles}
Let $T$ be a finite set of places of $\Q$ containing 
the archimedian place $\infty$ and the the prime $p$, and consider a residual representation
$$
\overline{\rho}: G_T \to \operatorname{GL}(2,\F_{p^{f}}).
$$
Let $M$ be a Galois stable subspace of $Ad(\overline{\rho})$ or $Ad(\overline{\rho})^*$. For each $q\in S,$ let $\mathcal{L}_q$ be a subspace of $H^1(G_q, M)$ with annihilator $\mathcal{L}_q^\perp \subseteq H^1(G_q, M^*)$ under the local pairing of Galois cohomology. Then 
\begin{IEEEeqnarray*}{ll}
 & \dim H^1_\mathcal{L} (G_T, M) - \dim H^1_{\mathcal{L}^\perp} (G_T, M^*) \\
= \quad & 
\dim H^0 (G_T, M) - \dim H^0 (G_T, M^*) 
+
\sum_{q\in T} \dim \mathcal{L}_q - \dim H^0(G_q, M).  
\end{IEEEeqnarray*}
\end{proposition}

\begin{proposition}[The global setting determines the structure of deformation rings] \label{R_N}
Let $S$ be a finite set of places of $\Q$, and let $G_S =\Gal(\Q_S/\Q)$, where $\Q_S$ is the maximal extension of $\Q$ unramified outside $S$. 
Consider a residual representation
$$
\overline{\rho}: G_S \to \operatorname{GL}(2, \F_{p^f}).
$$
Suppose for each $q\in S$ there is a class $\mathcal{C}_q$ of local $\overline{\rho}|_{G_q}$-deformations that are always liftable, and let $\mathcal{N}_q$ be the subspace of $H^1(G_q, Ad(\overline{\rho}))$
that preserves $\mathcal{C}_q.$
There exists a global, versal deformation ring
 $R_{(\mathcal{N}_q)_{q\in S}}$
 parametrizing $\overline{\rho}$-deformations $\rho$ with the property that 
 $\rho|_{G_q} \in \mathcal{C}_q$ for all $q\in S$.
 The structure of the ring is determined by the data 
 $$
(H^1_\mathcal{N}(G_{S}, Ad(\overline{\rho})),H^1_{\mathcal{N}^\perp}(G_{S}, Ad(\overline{\rho})^*) )
 $$
 of the global setting, in the sense that 
 $$
 R_{(\mathcal{N}_q)_{q\in S}} \simeq 
 \W[[T_1,\ldots,T_n]]/(r_1,\ldots, r_m)
 $$
 where 
 $$
n = \dim_{\F_{p^f}} H^1_\mathcal{N}(G_{S}, Ad(\overline{\rho})) 
$$
and 
$$
m \leq \dim H^1_{\mathcal{N}^\perp}(G_{S}, Ad(\overline{\rho})^*).
$$
Let $\mathfrak{m}_{(\mathcal{N}_q)_{q\in S}}$ be the maximal ideal of $R_{(\mathcal{N}_q)_{q\in S}}$.
The tangent space of 
$R_{(\mathcal{N}_q)_{q\in S}}$ is precisely the Selmer group, i.e. 
$$
\mathfrak{m}_{(\mathcal{N}_q)_{q\in S}} /(p,\mathfrak{m}_{(\mathcal{N}_q)_{q\in S}}^2) \simeq H^1_\mathcal{N}(G_{S}, Ad(\overline{\rho})). 
$$
\end{proposition}

\section{Constructing the residual representation and a smooth quotient of the local deformation ring at 7}
We will begin by defining a certain even, 2-dimensional residual representation in characteristic 2. Let $G_\Q = \Gal(\overline{\Q}/\Q)$ denote the absolute Galois group over $\Q$.
Let $K$ be the cyclic, totally real subfield of $\Q(\zeta_7)$ of degree $3$ over $\Q$. We work with the field $K$ for convenience.
By identifying a generator of $\Gal(K/\Q)$ with the element 
$$
\begin{pmatrix}
0 & 1 \\
1 & 1
\end{pmatrix}
$$
of order 3 in $\operatorname{GL}(2,\F_2)$, we obtain a Galois representation 
$$
G_\Q \to \operatorname{GL}(2, \F_2).
$$
Let $\alpha$ be a third root of unity in $\F_4$. It turns out to be more convenient to extend scalars to $\F_4$, and study the equivalent representation
$$
G_\Q \to \operatorname{GL}(2,\F_4)
$$
defined by 
$$
\overline{\rho} := 
\begin{pmatrix}
\psi & \\
& \psi^{-1}
\end{pmatrix},
$$
where $\psi: G_\Q \to \F_4$ is the character defined by identifying a generator of $\Gal(K/\Q)$ with $\alpha \in \F_4$.
The fixed field of $\ker \overline{\rho}$, which we denote
$\Q(\overline{\rho})$, is equal to $K$.  Since $K$ is totally real, $ \overline{\rho}$ is even, and since the order of 
$\Gal(K/\Q)$ is prime to 2, $ \overline{\rho}$ is tame. 

Let $G_\Q=\Gal(\overline{\Q}/\Q)$ denote the absolute Galois group over $\Q$. For a prime $p$, let $G_p$ denote $\Gal(\overline{\Q_p}/\Q_p).$
Let $G$ be either $G_\Q$ or $G_p$ for a prime $p$, and let $\kappa: G \to \F_4^\times$ be a character. The 1-dimensional $G$-module where $G$ acts as multiplication by $\kappa$ 
will be denoted $\F_4(\kappa)$, while the 1-dimensional $G$-module with trivial action will be denoted $\F_4$.
\begin{lemma} \label{Bockle}
Let $M$ be a vector space over $\F_4$ equipped with a $G$-action, where $G$ is either a local or global Galois group.  Then
$$
\text{\upshape{Hom}}_{\F_4}( M , \F_4 ) \simeq \text{\upshape{Hom}}_{\F_2}( M , \F_2 )
$$
as $\F_4[G]$-modules.
\end{lemma}
\begin{proof}
This Lemma is due to G. B\"{o}ckle and appears as Proposition 32 in \cite{HR}. 
We endow 
$\text{\upshape{Hom}}_{\F_2}( M , \F_2 )$ with the structure of a vector space over $\F_4$ by using the
$\F_4$-vector space structure on $M$.
\end{proof}
If $M$ is an $\F_4[G]$-module, we let $M^*$ denote the dual module of $M$. 
By Lemma \ref{Bockle}, we may think of $M^*$ as either an $\F_2[G]$-module or as an $\F_4[G]$-module.  

\begin{lemma} \label{character}
Let $\kappa: G \to \F_4^\times$ be a character, where $G$ can be either $G_\Q$ or $G_p$ for a prime $p$. Then
$$
\F_4(\kappa)^* = \F_4(\kappa^{-1}).
$$
\end{lemma}
\begin{proof}
It holds that $\F_4(\kappa)^*\simeq \F_4(\kappa^{-1}\chi)$ where $\chi$ is the cyclotomic character, and since we are in characteristic 2, $\chi$ is trivial. 
\end{proof}

The field $K$ cut out by $\overline{\rho}$ is totally ramified at 7 (and unramified everywhere else). The restriction $\overline{\rho}|_{G_7}$ factors through the tame quotient, which is topologically generated by $\sigma_7$ and $\tau_7$ subject to the relation $\sigma_7 \tau_7 \sigma_7^{-1}=\tau_7^{7}$, where 
$\sigma_7$ is a lift of Frobenius and $\tau_7$ generates inertia.
Letting $\alpha \in \F_4$ be the third root of unity defined above,
$$ \overline{\rho}|_{G_7}:
\sigma_7 \mapsto I, \quad \tau_7 \mapsto 
\begin{pmatrix}
\alpha & 0 \\
0 & \alpha^{-1}
\end{pmatrix}.
$$
Evidently,
$$
Ad(\overline{\rho}|_{G_7}) = 
\F_4(\psi^{-1}) 
\oplus \F_4
\oplus \F_4
\oplus \F_4(\psi),
$$
so $\dim H^0(G_7, Ad(\overline{\rho})) = 2$.
By Lemma \ref{character}, 
$$
Ad(\overline{\rho}|_{G_7})^* = 
\F_4(\psi^{}) 
\oplus \F_4
\oplus \F_4
\oplus \F_4(\psi^{-1}).
$$
We conclude that
$\dim H^0(G_7,Ad(\overline{\rho})^*)=2$, and by local duality of Galois cohomology $\dim H^2(G_7, Ad(\overline{\rho})) = 2$ as well.  By the local Euler-Poincar\'{e} characteristic, $\dim H^1(G_7, Ad(\overline{\rho})) = 4$. 
Let 
$\widetilde{\psi}$ be the Teichm\"{u}ller lift of $\psi$.
 Let $\mathcal{C}_7$ be the set of lifts of $\overline{\rho}|_{G_7}$ 
 of the form 
$$ 
\sigma_7 \mapsto 
\begin{pmatrix}
1+2x & 0 \\
0 & 1+2y
\end{pmatrix}
, \quad \tau_7 \mapsto 
\begin{pmatrix}
\widetilde{\alpha} & 0 \\
0 & \widetilde{\alpha}^{-1}
\end{pmatrix},
$$
where $x,y$ are arbitrary elements in $\W/2^k\W$ for $k\geq 1$, and $\widetilde{\alpha}\in \W$ is a lift of $\alpha\in \F_4$. 
By Lemma \ref{smoothquotients}, these two degrees of freedom give rise to a smooth quotient of $R_7$ which is isomorphic to a power series ring over $\W$ in two variables.
We define two 1-cocycles as follows:
$$
f_7^{(1)}: 
\sigma_7 \mapsto 
\begin{pmatrix}
1 & 0 \\
0 & 0
\end{pmatrix}
, \quad \tau_7 \mapsto 
\begin{pmatrix}
0 & 0 \\
0 & 0
\end{pmatrix},
\quad 
f_7^{(2)}: 
\sigma_7 \mapsto 
\begin{pmatrix}
0 & 0 \\
0 & 1
\end{pmatrix}
, \quad \tau_7 \mapsto 
\begin{pmatrix}
0 & 0 \\
0 & 0
\end{pmatrix}.
$$
To see that these functions define 1-cocycles, note that their corresponding lifts to the dual numbers send inertia to the identity, so the relation $\sigma_7 \tau_7 \sigma_7^{-1}=\tau_7 ^7$ is trivially satisfied.

\begin{lemma} \label{N_7independent}
The cohomology classes represented by $f_7^{(1)}$ and $f_7^{(2)}$ are linearly independent in $H^1(G_7, Ad(\overline{\rho}))$.
\end{lemma}

\begin{proof}
The $\F_4$-spaces
$
\begin{pmatrix}
* & 0 \\
0 & 0
\end{pmatrix}$ and 
$
\begin{pmatrix}
0 & 0 \\
0 & *
\end{pmatrix}
$
are both Galois stable subspaces of the adjoint whose sum is direct, and we observe that $f_7^{(1)}$ lives in 
$H^1(G_7, \begin{pmatrix}
* & 0 \\
0 & 0
\end{pmatrix})$, while $f_7^{(2)}$ lives in 
$H^1(G_7, \begin{pmatrix}
0 & 0 \\
0 & *
\end{pmatrix})$. Since cohomology commutes with direct sums, the lemma follows.
\end{proof}

Let $\mathcal{N}_7$ denote the subspace of $H^1(G_7, Ad(\overline{\rho}))$  that preserves $\mathcal{C}_7$.

\begin{lemma}\label{N_7}
The space $\mathcal{N}_7$ is spanned by $f_7^{(1)}$ and $f_7^{(2)}$ and is equal to space of unramified cohomology classes, i.e.
$$
\mathcal{N}_7 = H_{\operatorname{unr}}^1(G_7, Ad(\overline{\rho})),
$$
and 
$$
\dim \mathcal{N}_7 = \dim H^0(G_7, Ad(\overline{\rho})).
$$
\end{lemma}
\begin{proof}
There is a smooth quotient of the local deformation ring $R_7$ in two variables, and $\mathcal{C}_7$ consists of those deformations to $\W/2^n \W$ that factor through this smooth qoutient. It follows that 
$$
\dim \mathcal{N}_7 = 2. 
$$
The two unramified cohomology classes
$f_7^{(1)} $ and $f_l^{(2)}$ 
preserve $\mathcal{C}_7$. Therefore
$\mathcal{N}_7 = \operatorname{span}_{\F_4} \{ f_7^{(1)}, f_7^{(2)} \} $.
Since $\mathcal{N}_7 \subseteq H^1_{\operatorname{unr}}(G_7, Ad(\overline{\rho}))$, 
and 
$$
\dim H^1_{\operatorname{unr}}(G_l, Ad(\overline{\rho})) =2,$$ 
the lemma follows.
\end{proof}

\begin{lemma} \label{H^1(G_7)}
It holds that
$$
H^1(G_7,Ad(\overline{\rho})) = 
H^1(G_7, \begin{pmatrix}
* & 0 \\
0 & 0
\end{pmatrix})
\oplus
H^1(G_7, \begin{pmatrix}
0 & 0 \\
0 & *
\end{pmatrix}).
$$
\end{lemma}
\begin{proof}
Both of the cohomology groups in the direct sum on the right have dimension 2, while it follows from local duality and the local Euler-Poincar\'{e} characteristic that 
$$
H^1(G_7, 
\begin{pmatrix}
0 & * \\
0 & 0
\end{pmatrix}
) = 0, \quad
H^1(G_7, 
\begin{pmatrix}
0 & 0 \\
* & 0
\end{pmatrix}
) = 0.
$$
\end{proof}

\section{Smooth quotients at 2 and 
\texorpdfstring{$\infty$}{infinity}
} \label{2 and infinity}
Since Frobenius at 2 in $K$ has order 3,  $$\overline{\rho}|_{G_2}=
\begin{pmatrix}
\psi & \\
 & \psi^{-1}
\end{pmatrix}.
$$
Let $R_2$ denote the local deformation ring at 2. 
\begin{proposition} 
There is a smooth quotient of $R_2$ which is a power series ring in 5 variables, i.e. there is a surjective morphism
\[
\begin{tikzcd} 
R_2 \arrow[r, twoheadrightarrow] & \W [[T_1,\ldots,T_5]].
\end{tikzcd}
\]
\end{proposition}
\begin{proof}
Evidently,
$$
Ad(\overline{\rho}|_{G_2}) \simeq 
\F_4(\psi^{-1}) \oplus \F_4 \oplus \F_4 \oplus \F_4(\psi^{}).
$$
By Lemma \ref{character}, 
$$
Ad(\overline{\rho}|_{G_2})^* \simeq 
\F_4(\psi^{}) \oplus \F_4 \oplus \F_4 \oplus \F_4(\psi^{-1}).
$$
Clearly, $\dim H^0(G_2, Ad(\overline{\rho})) = 2$. By local duality,  $\dim H^2(G_2, Ad(\overline{\rho})) = 2$. The local Euler-Poincar\'{e} characteristic implies that $\dim H^1(G_2, Ad(\overline{\rho})=8$.
We conclude that
$$
R_2 \simeq \W[[z_1,z_2,z_3,z_4,z_5,z_6,z_7, z_8]]/(r_1,r_2)
$$
for some relations $r_1, r_2.$
In \cite{Paskunas}, V. Pa\v{s}k\={u}nas has computed an explicit presentation of the ring $R_2$, namely 
$$
R_2 \simeq \W[[z_1,z_2,z_3,z_4,z_5,x,y, u]]/((1+u)^2-1, 
z_5^2 - 2z_5 + xy).
$$
This ring has a smooth quotient in 5 variables. For example, 
there is a surjective morphism of the form  
\begin{IEEEeqnarray*}{rCl}
R_2 & \longrightarrow & \W[[z_1,z_2,z_3,z_4, x]], \\
(z_1,z_2,z_3,z_4,z_5,x,y, u) & \longmapsto &(z_1,z_2,z_3,z_4,0,x,0,0).
\end{IEEEeqnarray*}
We thank the referee for directing us to Pa\v{s}k\={u}nas' work. 
\end{proof}

\begin{proposition} \label{R_2}
    The class $\mathcal{C}_2$ of deformations
    $$
\rho^{(k)}: G_2 \to \operatorname{GL}(2, \W/2^k\W)
$$
for $k\in \Z_{\geq 1}$
of the form 
$$
\rho^{(k)}  
=
\begin{pmatrix}
    \widetilde{ \psi} \phi_1 &  f \\
    0 &  \widetilde{ \psi}^{-1} \phi_2
\end{pmatrix}
$$
where $\phi_1,\phi_2 \in \operatorname{Hom}(G_2, \W^\times)$ are such that 
$
\phi_i \equiv 1 \mod 2$ for 
$i=1,2$
are always liftable. More precisely, 
let $k\geq 1$, and let $\rho^{(k)} \in \mathcal{C}_2$. Then there exists 
    $$\rho^{(k+1)} : G_2 \to GL(2,\W/2^{k+1}\W)$$
    such that $\rho^{(k+1)} \in \mathcal{C}_2$ and 
    $$
    \rho^{(k+1)} \equiv \rho^{(k)} \mod 2^k.
    $$
\end{proposition}

\begin{proof}
The obstruction to lifting $\rho^{(k)} $ to a representation 
$$\rho^{(k+1)} : G_2 \to GL(2,\W/2^{k+1}\W)$$
of the form 
$$
\rho^{(k+1)} =
\begin{pmatrix}
    \widetilde{ \psi} \phi_1 & f^{(k+1)} \\
    0 & \phi_2
\end{pmatrix}
$$
is a cohomology class in 
$$
H^2(G_2, 
\begin{pmatrix}
    & * \\
    &
\end{pmatrix}
) = 0.
$$
(See Lemma 1 and Lemma 2 in \cite{SFM} for related computations.)
It follows that there exists a representation $\rho^{(k+1)}$ of the form above such that 
$\rho^{(k+1)} \equiv \rho^{(k)} \mod 2^k.$
By definition, $\rho^{(k+1)} \in \mathcal{C}_2$. 
\end{proof}

Note that 
$$
H^1(G_2, \begin{pmatrix}
    * & \\
    & *
\end{pmatrix}) = H^1( \widehat{\Z} \times \Z_2 \times \Z/2\Z, \F_4 \oplus \F_4)
$$
while the inflation map gives an injection
$$
H^1( \widehat{\Z} \times \Z_2, \F_4\oplus \F_4)  \to H^1( \widehat{\Z} \times \Z_2 \times \Z/2\Z, \F_4 \oplus \F_4). 
$$
Here, we regard $\widehat{\Z} \times \Z_2$ as a quotient of $\widehat{\Z} \times \Z_2 \times \Z/2\Z$.
The space $H^1( \widehat{\Z} \times \Z_2, \F_4 \oplus \F_4)$ is generated by
$$ f_2^{(1)}:
\sigma \mapsto 
\begin{pmatrix}
1 & 0 \\ 
0 & 0
\end{pmatrix},
\quad 
\tau \mapsto
\begin{pmatrix}
0 & 0 \\ 
0 & 0
\end{pmatrix},
\quad 
f_2^{(2)}:
\sigma \mapsto 
\begin{pmatrix}
0 & 0 \\ 
0 & 0
\end{pmatrix},
\quad 
\tau \mapsto
\begin{pmatrix}
1 & 0 \\ 
0 & 0
\end{pmatrix}, 
$$
and 
$$ 
f_2^{(3)}:
\sigma \mapsto 
\begin{pmatrix}
0 & 0 \\ 
0 & 1
\end{pmatrix},
\quad 
\tau \mapsto
\begin{pmatrix}
0 & 0 \\ 
0 & 0
\end{pmatrix}, 
\quad
f_2^{(4)}:
\sigma \mapsto 
\begin{pmatrix}
0 & 0 \\ 
0 & 0
\end{pmatrix},
\quad 
\tau \mapsto
\begin{pmatrix}
0 & 0 \\ 
0 & 1
\end{pmatrix}.
$$

\begin{lemma}
    If $\rho^{(k)}\in \mathcal{C}_2$ and 
    $
    f \in \operatorname{Hom}( \widehat{\Z} \times \Z_2, 
    \begin{pmatrix}
        * & \\
        & *
    \end{pmatrix}
    )
    $
    then     
    $$
(1+2^{k-1}f)\rho^{(k)} \in \mathcal{C}_2.
    $$
\end{lemma}
\begin{proof}
There are additive characters $\beta_1,\beta_2 \in \operatorname{Hom}(\widehat{\Z}\times \Z_2, \F_4)$ such that
$$
f=
\begin{pmatrix}
  \beta_1 & \\
  &  \beta_2
\end{pmatrix}.
$$
For each $i =1, 2$,
$$
\gamma_i = (1+2^{k-1})\beta_i : \widehat{\Z}\times \Z_2 \to (\W/2^{k}\W)^\times 
$$
is a multiplicative character in $\operatorname{Hom}(\widehat{\Z}\times \Z_2, (\W/2^{k}\W)^\times )$
Evidently, there exist characters 
$$
\widetilde{\gamma}_1, \widetilde{\gamma}_2: \widehat{\Z}\times \Z_2 \to \W^\times  
$$
such that $$\widetilde{\gamma}_i \equiv \gamma_i \mod 2^k.$$
Computing modulo $2^k$,
\begin{IEEEeqnarray*}{rCl}
(I+2^{k-1}f)\rho^{(k)}
& =&
\begin{pmatrix}
    \widetilde{\psi}\phi_1(1+2^{k-1}\beta_1) & * \\
     & \phi_2 (1+2^{k-1}\beta_2)
\end{pmatrix} \\
& = &
\begin{pmatrix}
    \widetilde{\psi}\phi_1 \gamma_1 & * \\
     & \widetilde{\psi}^{-1} \phi_2 \gamma_2
\end{pmatrix} \\
& = &
\begin{pmatrix}
    \widetilde{\psi}\phi_1 \widetilde{\gamma}_1 & * \\
     &\widetilde{\psi}^{-1} \phi_2 \widetilde{\gamma}_2
\end{pmatrix} \in \mathcal{C}_2.  
\end{IEEEeqnarray*}
\end{proof}

\begin{corollary} \label{N_2}
    The 5-dimensional subspace
    $$
    \mathcal{N}_2 = 
    H^1(G_2, 
    \begin{pmatrix}
        & * \\
     & \end{pmatrix})
    \oplus 
  H^1( \widehat{\Z} \times \Z_2, 
    \begin{pmatrix}
        * & \\
        & *
    \end{pmatrix}
    )
    $$
    of $H^1(G_2, Ad(\overline{\rho}))$ preserves $\mathcal{C}_2.$
\end{corollary}

Next, we consider the infinite place. We will always require our representations be trivial when restricted to $G_\infty$. This forces us to take $\mathcal{N}_\infty =0$. Note that
$\overline{\rho}|_{G_\infty}$ is trivial because $K$ is totally real. Hence $$Ad(\overline{\rho}|_{G_\infty}) = \F_4^4,$$ 
so $\dim H^0(G_\infty, Ad(\overline{\rho}))=4$. 

\section{Thickening the image along the diagonal}
Let $l$ be a prime that remains prime in $K$ and which is congruent to 1 mod 8. The smallest such prime is $l=17$.
Let $\chi_l: G_\Q \to \F_2$ denote the quadratic character corresponding to $\Q(\sqrt{l})$.
Consider the following lift of $\overline{\rho}$ to  $\operatorname{GL}(2,\W/4\W)$:
$$
\xi_2 =
\begin{pmatrix}
\widetilde{\psi}(1+2 \chi_l)& \\
& \widetilde{\psi}^{-1}
\end{pmatrix}.
$$
We compute $\xi_2|_{G_l}$. 
Since $l$ is totally ramified in $\Q(\sqrt{l})$,
it holds that 
$$\chi_l(\sigma_l) = 0, \quad 
\chi_l(\tau_l) = 1.$$
Hence,
$$\xi_2|_{G_l}:
\sigma_l \mapsto 
\begin{pmatrix}
 \widetilde{\alpha} & \\
 & \widetilde{\alpha}^{-1}
\end{pmatrix},
\quad
\tau_l \mapsto
\begin{pmatrix}
(1+2) & \\
& 1
\end{pmatrix}
=
\begin{pmatrix}
-1 & \\
& 1
\end{pmatrix}.
$$
\begin{lemma}
$\dim H^0(G_l, Ad(\overline{\rho})) =2.$ 
\end{lemma}
\begin{proof}
Since $l$ remains prime in $K$,
$$
\overline{\rho}|_{G_l}: \sigma_l \mapsto 
\begin{pmatrix}
 \alpha & \\
 & \alpha^{-1}
\end{pmatrix}, 
\quad \tau_l \mapsto I,
$$
The statement follows.
\end{proof}
Consider the following class $\mathcal{C}_l$ of lifts of $\overline{\rho}|_{G_l}$ to $\operatorname{GL}(2,\W/2^k\W)$ for some $k \geq 1$ for the form
$$ \pi_l:
 \sigma_l \mapsto 
\begin{pmatrix}
  \widetilde{\alpha}(1+2x) & \\
 &  \widetilde{\alpha}^{-1}(1+2y)
\end{pmatrix}, 
\quad \tau_l \mapsto 
\begin{pmatrix}
-1 & \\
& 1
\end{pmatrix}.
$$
Since $\pi_l(\tau_l)$ has order dividing 2  and $l$ is odd,
$\pi_l(\tau_l)^l = \pi_l(\tau_l)$, and since the images of $\sigma_l$ and $\tau_l$ under $\pi_l$ commute, we see that $\pi_l$ respects the relation $\sigma_l \tau_l \sigma_l^{-1}=\tau_l^l$. Evidently, $\pi_l$ is always liftable. 
Let $\mathcal{N}_l$ be the subspace of $H^1(G_l,Ad(\overline{\rho}))$
that
preserves $\mathcal{C}_l$.
\begin{lemma} \label{N_l}
$$
\mathcal{N}_l = H^1_{\operatorname{unr}}(G_l, Ad(\overline{\rho})),
$$
and 
$$
\dim \mathcal{N}_l = \dim H^0(G_l, Ad(\overline{\rho})) = 2.
$$
\end{lemma}
\begin{proof} The proof is almost identical to the proof of Lemma \ref{N_7}.
By Lemma \ref{smoothquotients}, there is a smooth quotient of the local deformation ring $R_l$ that is a power series ring in 2 variables, and $\mathcal{C}_l$ consists of those deformations to $\W/2^k \W$ that factor through this smooth quotient. It follows that 
$$
\dim \mathcal{N}_l = 2. 
$$
The two unramified cohomology classes
$$
f_l^{(1)}: 
\sigma_l \mapsto 
\begin{pmatrix}
1 & 0 \\
0 & 0
\end{pmatrix}
, \quad \tau_l \mapsto 
\begin{pmatrix}
0 & 0 \\
0 & 0
\end{pmatrix},
\quad 
f_l^{(2)}: 
\sigma_l \mapsto 
\begin{pmatrix}
0 & 0 \\
0 & 1
\end{pmatrix}
, \quad \tau_l \mapsto 
\begin{pmatrix}
0 & 0 \\
0 & 0
\end{pmatrix}
$$
preserve $\mathcal{C}_l$. Therefore
$\mathcal{N}_l = \operatorname{span}_{\F_4} \{ f_l^{(1)}, f_l^{(2)} \} $.
Since $\mathcal{N}_l \subseteq H^1_{\operatorname{unr}}(G_l, Ad(\overline{\rho}))$, 
and 
$$
\dim H^1_{\text{unr}}(G_l, Ad(\overline{\rho})) =2,$$ 
the lemma follows.
\end{proof}
By local duality, $\dim H^2(G_l,Ad(\overline{\rho})) =2 $ and by the local Euler-Poincar\'{e} characteristic, $\dim H^1(G_l,Ad(\overline{\rho})) = 4$. 
\begin{lemma} \label{H^1(G_l)}
It holds that 
    $$
H^1(G_l,Ad(\overline{\rho}))=H^1(G_l,
\begin{pmatrix}
 *&\\
 &*
\end{pmatrix}
)
$$
\end{lemma}
\begin{proof}
Using the local Euler-Poincar\'{e} characteristic, it is easy to check that 
$$
H^1(G_l,
\begin{pmatrix}
 & *\\
 *&
\end{pmatrix}
)
=0.
$$
\end{proof}

Next, we will verify that $\xi_2|_{G_q}$ belongs to $\mathcal{C}_q$ for each $q\in \{ \infty, 2, 7, l \}$. Since $\Q(\sqrt{l})$ is totally real, $\xi_2|_{G_\infty} \in \mathcal{C}_\infty$.  It is clear that $\xi_2|_{G_l} \in \mathcal{C}_l$. 
Next, we consider 7. 
Recall that  $K=\Q(\widetilde{\psi})$ is totally ramified at 7 so that
$$
\widetilde{\psi}(\sigma_7)=1, \quad 
\widetilde{\psi}(\tau_7)=\widetilde{\alpha}.
$$
Since 7 is unramified in $\Q(\sqrt{l})$, 
$$ \chi_l(\tau_7)=0.$$
Therefore,
$$ \xi_2|_{G_7}:
\sigma_7 \mapsto 
\begin{pmatrix}
1+2 \chi_l(\sigma_7) & 0 \\
0 & 1
\end{pmatrix}
, \quad \tau_7 \mapsto 
\begin{pmatrix}
 \widetilde{\alpha} & 0 \\
0 &  \widetilde{\alpha}^{-1}
\end{pmatrix},
$$
from which we conclude that 
$\xi_2|_{G_7} \in \mathcal{C}_7$. Finally, we consider 2. 
The fact that $2$ remains prime in $K$ implies that 
$
\widetilde{\psi}|_{G_2} 
$
has order 3. 
The fact that $l \equiv 1 \mod 8$ implies that 2 splits completely in $\Q(\sqrt{l})$, and therefore
$$
\chi_l|_{G_2} = 0.
$$
We conclude that
$$ \xi_2|_{G_2} =
\begin{pmatrix}
\widetilde{\psi} & \\
& \widetilde{\psi}^{-1}
\end{pmatrix} \in \mathcal{C}_2.
$$
\section{Thickening the image off the diagonal and irreducibility} \label{irred}
In \cite{HR}, a type of auxiliary primes called \emph{trivial primes} is defined. Furthermore, trivial primes are used in the relative deformation theory developed in \cite{FKP1} and \cite{FKP}. Below, we adapt the notion of trivial primes to the setting of this paper, where the residual characteristic is $p=2$ and the field $\Q(\overline{\rho})$ cut out by $\overline{\rho}$ is the totally real field $K$ of degree 3 over $\Q$.

\begin{definition}
We will say that a prime $p$ is trivial if
\begin{enumerate} \label{deftrivialprimes}
    \item $p \equiv 3 \mod 4$,
    \item $p$ splits completely in $K.$
\end{enumerate} 
\end{definition}
Next, we introduce the \emph{governing field} of $K$. Governing fields can be used to study $\Z/p\Z$-extensions of $K$ with prescribed tame ramification. Here, we will work with the case $p=2$.
For more details on governing fields, see \cite{gras}.

The governing field of $K$ is the extension $K(\sqrt{V_\emptyset})$, where
$$
V_\emptyset = \{ x \in K^\times | (x) = J^2 \text{ for some fractional ideal } J \}.
$$
There is an exact sequence
$$
1 \to U_K / U_K^{ 2} \to V_\emptyset/K^{\times 2} \to \cl_K[2] \to 1.
$$
Since $\cl_K$ is trivial, the governing field is $K(\sqrt{U_K})$. The unit group of $K$ has rank 2, and $$
U_K / U_K^{ 2} \simeq (\Z/2)^3.
$$
We conclude that    
$$
\Gal(K(\sqrt{U_K})/K) \simeq \F_2 \oplus U_2 $$
as $\Gal(K/\Q)$-modules, where $U_2$ is the irreducible 2-dimensional representation of $\Z/3$ over $\F_2$, and $\F_2$ is the trivial 1-dimensional representation. This follows from the classification of linear representations of $\Z/3\Z$ over the field $\F_2$ (see \cite{rep}).

By Chebotarev's theorem, there are infinitely many \emph{degree one} primes in $K$ whose Frobenius is equal to a given element in the Galois group of an abelian extension of $K$. Therefore, we may choose a prime $p$ such that 
\begin{enumerate}
    \item $p$ splits completely in $K$,
    \item Writing 
    $$p O_K=\mathfrak{p}_1 \mathfrak{p}_2 \mathfrak{p}_3,$$ 
    it holds that 
    $$\sigma_{\p_1} = (1,0)\in \F_2 \oplus U_2.$$
\end{enumerate}
Now, $\sigma_{\p_2}$ and $\sigma_{\p_3}$ are in the $\Gal(K/\Q)$-orbit of $\sigma_{\p_1}$, so the span of these three Frobenius automorphisms is just $\F_2$.
Let $L$ denote the subfield of the governing field fixed by $U_2$. Evidently, $L$ descends to $\Q$; indeed it is easy to see that $L$ is just the base change of $\Q(\sqrt{-1})$ to $K$. In particular, the fact that the Frobenius automorphisms above $p$ are nontrivial in $L$ forces $
p \equiv 3 \mod 4
$. Hence, $p$ is a trivial prime according to Definition \ref{deftrivialprimes}.
For any set of places $S$ in $K$, let 
$K_{S}$ be the maximal 2-extension of $K$ unramified outside $S$. If $S$ is stable under $\Gal(K/\Q),$ then $K_S$ is Galois over $\Q.$
Let $S_p=\{\mathfrak{p}_1,\mathfrak{p}_2,\mathfrak{p}_3\}$. Since we are in a tame setting, it follows
from \cite{gras} (equation (2.2), p. 3) that 
$$
\dim \frac{H^1(G_{S_p }, \Z /2)}{H^1(G_\emptyset, \Z /2)}
= \#S_p - \dim (\text{span} \{ \sigma_v |v\in S_p \}) =3-1= 2. 
$$
Since the class number of $K$ is prime to 2, $H^1(G_\emptyset, \Z /2)$ is trivial.
It follows that the maximal 2-elementary abelian extension of $K$ ramified at the primes above $p$, denoted $K^{(p)}$, has 
$$M:=\Gal(K^{(p)}/K) = \F_2^2.$$ Note that $K^{(p)}$ is Galois over $\Q$, and since the order of $M$ is prime to 3, $$\Gal(K^{(p)}/\Q) = M \rtimes \Gal(K/\Q).$$
We claim that $M$ is isomorphic to $U_2$ as a $\Gal(K/\Q)$-module. If not, then $M$ would be a sum of two trivial 1-dimensional representations over $\F_2$, and hence $K^{(p)}$ would descend to $\Q$, but there are no biquadratic extensions of $\Q$ ramified only at one prime $p$. 

We regard $\Gal(K/\Q)$ as a subgroup of $\operatorname{GL}(2,\F_2)$ with generator
$$
g=\begin{pmatrix}
0 & 1 \\
1 & 1
\end{pmatrix}.
$$
The characteristic polynomial of $g$ is $x^2+x+1$, so $g$ has eigenvalues $\alpha$ and $\alpha^{-1}$, where $\alpha$ is a primitive third root of unity in $\F_4$. The eigenvector for $\alpha$ is  
$$v_1= \begin{pmatrix}
\alpha^{-1} \\
1
\end{pmatrix},$$
and the eigenvector for $\alpha^{-1}$ is  
$$v_2= \begin{pmatrix}
\alpha^{} \\
1
\end{pmatrix}.$$
Let $e_1 =  \begin{pmatrix}
1 \\
0
\end{pmatrix} \in M$ and $e_2 =  \begin{pmatrix}
0 \\
1
\end{pmatrix} \in M.$
Let $$\iota: M \to \F_4 v_1 \oplus \F_4 v_2 $$ be the inclusion of $M$ as a subset. Then 
$$
\iota: e_1 \mapsto \begin{pmatrix}
1 \\
0
\end{pmatrix}
=
\begin{pmatrix}
\alpha^{-1} \\
1
\end{pmatrix}
+
\begin{pmatrix}
\alpha^{} \\
1
\end{pmatrix}, \quad 
\iota: e_2 \mapsto \begin{pmatrix}
0 \\
1
\end{pmatrix}
=
\alpha \begin{pmatrix}
\alpha^{-1} \\
1
\end{pmatrix}
+ \alpha^{-1}
\begin{pmatrix}
\alpha^{} \\
1
\end{pmatrix}.
$$
It is trivial to see that $\iota$ is $\F_2$-linear and $\Gal(K/\Q)$-equivariant. 
Let 
$$\pi_1: \F_4 v_1 \oplus \F_4 v_2 \to \F_4$$ 
by the projection of a vector in $\F_4 v_1 \oplus \F_4 v_2$ onto the first coordinate. Then $\pi_1$ is a $\Gal(K/\Q)$-morphism to $\F_4 (\psi)$:
$$
\pi_1(g\cdot(a v_1 + bv_2)) = \pi_1(\alpha a v_1 + \alpha^{-1} b v_2) = \alpha a = \psi(g) \pi_1(a v_1 + bv_2).
$$
Similarly, $\pi_2$ is a $\Gal(K/\Q)$-morphism to $\F_4(\psi^{-1})$.
Let $\varphi_1=\pi_1 \circ \iota $. This defines an $\F_2$-linear, $\Gal(K/\Q)$-equivariant map $M\to \F_4(\psi)$. Since 
$$
\varphi_1(e_1) = \pi_1 (1\cdot v_1 + v_2) = 1,
\quad 
\varphi_1(e_2) = \pi_1 ( \alpha v_1 + \alpha^{-1} v_2) = \alpha,
$$
we see that $\varphi_1$ is surjective, and hence an isomorphism.

By similar reasoning, there is a $\Gal(K/\Q)$-isomorphism $\varphi_2 =\pi_2 \circ \iota :M \to \F_4(\psi^{-1})$. In Lemma \ref{v,w} below, we summarize our findings of this section so far.

\begin{lemma}\label{v,w} In the above notation, for any trivial prime $p$
such that the
Frobenius automorphisms at primes above $p$ 
equals $(1,0)$ in $\F_2 \oplus U_2,$
one can find an isomorphism of
$\Gal(K^{(p)}/K)$ with $ \F_4(\psi)$
and one with $ \F_4(\psi^{-1})$ as $\Gal(K/\Q)$-modules.
Once and for all, we choose a trivial prime $v$ such that 
$$\Gal(K^{(v)}/K) = \F_4(\psi^{-1})$$ 
as $\Gal(K/\Q)$-modules, and such that $v$ remains prime in $\Q(\sqrt{l})$. Then, we choose another trivial prime $w$ for which there  is a field $K^{(w)}$ such that 
$$\Gal(K^{(w)}/K) = \F_4(\psi),$$
such that $w$ also remains prime in $\Q(\sqrt{l})$, and such that
 $w$ splits completely in $K^{(v)}$.
\end{lemma}

If we allow our residual representation to have coefficients in $\F_4$, then we can express $g$ as a matrix with respect to the basis $v_1,v_2$ defined above. It takes the form
$$
\begin{pmatrix}
\alpha & 0 \\
0 & \alpha^{-1}
\end{pmatrix}.
$$
Then we may identify $\F_4(\psi)$ with the subspace 
$$
\begin{pmatrix}
0 & 0 \\
* & 0
\end{pmatrix}
$$
of the adjoint with coefficients in $\F_4$, and we may identify $\F_4(\psi^{-1})$ with the subspace 
$$
\begin{pmatrix}
0 & * \\
0 & 0
\end{pmatrix}.
$$
If $f$ is a 1-cocycle, we let $f|_K$ denote the restriction of $f$ to $\ker\overline{\rho}$. Note that $f|_K$ is a homomorphism.

\begin{proposition} \label{globalclasses}
The number fields $K^{(v)}$ and $K^{(w)}$ give rise to cohomology classes
$f^{(v)} \in H^1(G_S, 
\begin{pmatrix}
& * \\
&
\end{pmatrix}
)$ and $f^{(w)} \in H^1(G_S, 
\begin{pmatrix}
&  \\
* &
\end{pmatrix}
)$.
\end{proposition}

\begin{proof} Note that
$$\Gal(K^{(v)}/\Q) \simeq \F_4(\psi^{-1}) \rtimes \Gal(K/\Q),$$ which again can be realized as a subgroup of $$\operatorname{GL}(2,\F_4[\varepsilon])=(I+\varepsilon Ad(\overline{\rho}))\rtimes \operatorname{GL}(2,\F_4).$$
Therefore the projection $$G_\Q \to \Gal(K^{(v)}/\Q)$$ defines a lift of $\overline{\rho}$ to the dual numbers, which corresponds to a 1-cocycle
$f^{(v)} \in H^1(G_\Q, \F_4(\psi^{-1}))$. The field $K^{(v)}$ is then equal to the field $K_{f^{(v)}}$, ie. the fixed field of $\ker (f^{(v)}|_K)$. 
Analogously, $K^{(w)}$ defines a 1-cocycle $f^{(w)}\in H^1(G_\Q, \F_4(\psi))$.
\end{proof}

\begin{lemma}
The fields $K^{(v)}$ and $K^{(w)}$ are linearly disjoint.
\end{lemma}
\begin{proof}
Let $L=K^{(v)} \cap K^{(w)}$. Since the two fields are ramified at distinct primes, $L$ cannot be equal to either of $K^{(v)}$ or $K^{(w)}$. Suppose $[L:K]=2$. Since there are no unramified 2-extensions of $K$, some place of $K$ must ramify in $L$. Since $L$ is contained in $K^{(v)}$, this place must be a prime in $K$ above $v$. But since $L$ is contained in $K^{(w)}$, the ramified place must divide $w$. This contradiction shows that $L=K$.
\end{proof}
As usual, let $\psi$ be the character that sends a generator of $\Gal(K/\Q)$ to $\alpha \in \F_4$, and let $\widetilde{\psi}$ denote the Teichmuller lift of $\psi$ to $\W$. Recall that  
$$\xi_2=
\begin{pmatrix}
\widetilde{\psi}(1+2\chi_l) & 0 \\
0 & \widetilde{\psi}^{-1}
\end{pmatrix}
\mod 4
$$ 
is a lift of $\overline{\rho}$ to $\W / 4\W $.
Let
$$
\rho_2 = (I +2(f^{(v)}+f^{(w)}))\xi_2.
$$
\begin{lemma}
$$\ker \rho_2 = \ker f^{(v)}|_K \cap \ker f^{(w)}|_K\cap \ker \chi_l.$$
\end{lemma}
\begin{proof}
Clearly
$$\ker \rho_2 \subseteq \ker \overline{\rho}
= \ker\widetilde{\psi}.$$ 
It follows that $g\in \ker\rho_2$ if and only if 
$$
I+2 \Big( f^{(v)}(g)+f^{(w)}(g)+
\begin{pmatrix}
\chi_l(g) & \\
&
\end{pmatrix}
\Big)  =I
$$
in $GL(2, \W / 4 \W)$, which is equivalent to 
$$f^{(v)}(g)+f^{(w)}(g)
+
\begin{pmatrix}
\chi_l(g) & \\
&
\end{pmatrix}
= 0$$ 
in $M_2(\F_4)$.
Since all the cohomology classes in this sum are independent, the sum equals zero if and only if  
$f^{(v)}(g)=0, f^{(w)}(g)=0,$ and $\chi_l(g)=0.$
\end{proof}
In conclusion, 
$$\Gal(K(\rho_2)/K) = \Gal(K^{(v)}K^{(w)}K(\sqrt{l})/K) \supseteq \F_4(\psi^{-1})\oplus \F_4(\psi^{}).$$

\begin{definition} \label{reducible}
    Let $R$ be a ring and let $M = R^2$ (the free module of rank 2 over $R$).
We will say that a representation $\rho: G_\Q \to \operatorname{GL}(2,R)$ is \emph{reducible} if
there exist free $R$-modules $M_1$ and $M_2$ of rank one, where $M_1$ is $G_\Q$-stable, such that
$$
M= M_1 \oplus M_2.
$$
In this case, $M_1=R x_1$ and $M_2=R x_2$ for $x_1,x_2 \in M$, where $\{x_1,x_2\}$ is an $R$-basis for $M$. With respect to this basis,
$$
\rho =
\begin{pmatrix}
\chi & * \\
0 & \chi' 
\end{pmatrix}
$$ for some characters $\chi, \chi': G_\Q \to R^\times.$
If $\{e_1, e_2\}$ is another $R$-basis for $M$, the change of basis matrix $A$ (from $\{e_1,e_2\}$ to $\{x_1,x_2\}$) is an element of
$\operatorname{GL}(2,R)$ (see \cite{Lang} p. 508 Proposition 3.1).
We say that  $\rho$ is \emph{irreducible} if it is not reducible according to the definition above.
\end{definition}

\begin{proposition} \label{irreduciblelifts}
Any lift of $\rho_2$ to $\W$ is irreducible.
\end{proposition}
\begin{proof}
Let $\rho$ be a lift to $\W$. Suppose $\rho$ is reducible. 
According to Definition \ref{reducible}, there is some
$A\in GL(2,\W)$ such that 
$$
\rho':=A\rho A^{-1} = 
\begin{pmatrix}
\chi & * \\
0 & \chi'
\end{pmatrix},$$

while $$
\ker \rho = \ker \rho',
$$
and hence $\Q(\rho)=\Q(\rho')$. 
Similarly, if $\rho_n$ denotes reduction of $\rho \mod 2^n$, then $\ker\rho_n = \ker \rho'_n$.
In particular, $\overline{\rho}$ and $\overline{\rho'}$ cut out the same field $K$. Letting $\overline{*}$ denote reduction of $* \mod 2$, we have
$$
\begin{pmatrix}
\overline{\chi} & \overline{*} \\
0 & \overline{\chi'}
\end{pmatrix}
=
\overline{A} 
\begin{pmatrix}
\psi & 0 \\
0 & \psi^{-1}
\end{pmatrix}
\overline{A}^{-1}
=
\begin{pmatrix}
ad \psi + bc\psi^{-1} & ab(\psi+\psi^{-1})\\
cd(\psi+\psi^{-1}) & bc\psi +ad \psi^{-1} 
\end{pmatrix}\frac{1}{ad-bc},
$$
where
$$
\overline{A}=
\begin{pmatrix}
a&b \\
c&d
\end{pmatrix}.
$$

It follows that $cd(\psi+\psi^{-1})=0$ and therefore $cd=0$. 
If $c=0$, then 
$$
\overline{\rho'} = 
\begin{pmatrix}
\psi & (b/d)(\psi+\psi^{-1}) \\
0 & \psi^{-1}
\end{pmatrix},
$$
while the zero in the lower left corner of $\rho'$ implies that 
$\ker \overline{\rho}/ \ker \rho'_2$
will be contained in 
$$\begin{pmatrix}
* & *\\
0 & *
\end{pmatrix}.$$
This implies that $\F_4 (\psi^{-1} / \psi) = \F_4(\psi)$ does \emph{not} occur in the Jordan-H\"{o}lder sequence for $\ker \overline{\rho}/ \ker \rho'_2$ as a $\Gal(K/\Q)$-module. But 
$$
\ker \overline{\rho} /\ker \rho'_2 
=
\ker \overline{\rho} /\ker \rho_2 
\supseteq
\F_4(\psi)\oplus \F_4(\psi^{-1})
$$
as  $\Gal(K/\Q)$-modules, and hence we have reached a contradiction.
If $d=0$, 
$$
\overline{\rho'} = 
\begin{pmatrix}
\psi^{-1} & (a/c)(\psi+\psi^{-1}) \\
0 & \psi^{}
\end{pmatrix}.$$
By the same reasoning as in the previous case, we may then conclude that $\F_4 (\psi^{} / \psi^{-1}) = \F_4(\psi^{-1})$ does not occur in the Jordan-H\"{o}lder sequence for $\ker \overline{\rho}/ \ker \rho_2$, and again we have a contradiction. We conclude that $\rho$ is irreducible. 
\end{proof}
Note that the occurrence of the two cases in the proof above is exactly why we need to thicken the image by acting on it with the two distinct cohomology classes $f^{(v)}$ and $f^{(w)}$; one of them alone would not have been sufficient to force the lift to be irreducible.

\section{Generic smoothness at trivial primes} \label{trivialprimes}
Let $p$ be a trivial prime. 
Let $\mathcal{C}_p$ be the set of deformations of $\overline{\rho}|_{G_p}$
of the form
$$ \pi_p:
\sigma_p \mapsto 
\begin{pmatrix}
  p(1+z) & x \\
  0 & 1+z
\end{pmatrix}, \quad \tau_p \mapsto
\begin{pmatrix}
  1 & y \\
  0 & 1
\end{pmatrix}
$$
for some $n \geq 1$  and some $x,y,z \in \W/2^n\W$ all divisible by 2.
Since
$$
\begin{pmatrix}
p(1+z) & x \\
0 & 1+z
\end{pmatrix}
\begin{pmatrix}
1 & y \\
0 & 1
\end{pmatrix}
\begin{pmatrix}
p(1+z) & x \\
0 & 1+z
\end{pmatrix}^{-1}
=
\begin{pmatrix}
1 & yp \\
0 & 1
\end{pmatrix},
$$
$\pi_p$ respects the relation
$\sigma_p \tau_p \sigma_p^{-1}=\tau_p^p$ and hence defines a homomorphism on the tame quotient of the local Galois group at $p$. Also, $\pi_p$ is a lift of $\overline{\rho}|_{G_p}$ which is just the trivial representation, since $p$ splits completely in $K$. Note that $\pi_p$ is always liftable to characteristic zero. It is straightforward to check that the following three cohomology classes preserve $\mathcal{C}_p$:
$$
f_p^{(1)}: 
\sigma_p \mapsto 
\begin{pmatrix}
0 & 1 \\
0 & 0
\end{pmatrix}
, \quad \tau_p \mapsto 
\begin{pmatrix}
0 & 0 \\
0 & 0
\end{pmatrix},
\quad 
f_p^{(2)}: 
\sigma_p \mapsto 
\begin{pmatrix}
0 & 0 \\
0 & 0
\end{pmatrix}
, \quad \tau_p \mapsto 
\begin{pmatrix}
0 & 1 \\
0 & 0
\end{pmatrix}
$$
and
$$
f_p^{(3)}: 
\sigma_p \mapsto 
\begin{pmatrix}
1 & 0 \\
0 & 1
\end{pmatrix}
, \quad \tau_p \mapsto 
\begin{pmatrix}
0 & 0 \\
0 & 0
\end{pmatrix}.
$$
By an argument similar to the proof of Lemma \ref{N_7independent}, $f_p^{(1)}$ and $f_p^{(3)}$ are linearly independent. Next, note that $f_p^{(2)}$ is not in the span on $f_p^{(1)}$ and $f_p^{(3)}$, since any element in this span cuts out an unramified extension, while $f_p^{(2)}$ cuts out a ramified extension. Hence we have found a 3-dimensional subspace of $H^1(G_p, Ad)$ that preserves $\mathcal{C}_p$.

The main goal of the rest of this section is to show that once we lift our representation beyond mod 4, the subspace of $H^1(G_p, Ad(\overline{\rho}))$ that preserves the desired form locally at $p$ can be enlarged by one dimension. 
Similar computations were done in characteristic different from 2 in \cite{HR}. Here, we verify that such computations carry through in characteristic 2. 

Let $\mathcal{C}_p^{\operatorname{unr}}$ be the subset of $\mathcal{C}_p$  of all deformations which are unramified mod 4. Let $\mathcal{C}_p^{\operatorname{unr}, \geq 3}$ be the subset of $\mathcal{C}_p^{\operatorname{unr}}$ consisting of deformations to $\operatorname{GL}(2,\W/2^n\W)$ for $n\geq 3$ and which are all deformations of the same fixed mod 4 deformation of $\overline{\rho}$.
Let $g_p^{\operatorname{unr}} \in H^1(G_p, Ad(\overline{\rho}))$ be defined by
$$
\sigma_p \mapsto 
\begin{pmatrix}
  0 & 0 \\
  1 & 0
\end{pmatrix}, \quad \tau_p \mapsto
\begin{pmatrix}
  0 & 0 \\
  0 & 0
\end{pmatrix}.
$$
\begin{proposition} Let $\rho_n \in \mathcal{C}_p^{\operatorname{unr}, \geq 3}$ have coefficients in $\W/2^n\W$ and let $a\in \F_4$. Then $(I+2^{n-1} a g_p^{\operatorname{unr}})\rho_n \in \mathcal{C}_p^{\operatorname{unr}, \geq 3}$.
\end{proposition}
\begin{proof}
Since $\rho_n \in \mathcal{C}_p^{\operatorname{unr}, \geq 3}$,
$$
\rho_n: \sigma_p \mapsto 
\begin{pmatrix}
p(1+z) & x \\
0 & 1+z
\end{pmatrix},\quad
\tau_p \mapsto 
\begin{pmatrix}
1 & y \\
0 & 1
\end{pmatrix}
$$
for some $x,y,z$ in $\W/2^n\W $ such that $2$ divides $x,y,z$ and $4$ divides $y$.
We compute
\begin{IEEEeqnarray*}{rCl}
(I+2^{n-1} a g_p^{\operatorname{unr}}(\sigma_p)) \rho_n (\sigma_p) 
&=&
\begin{pmatrix}
  1 & 0 \\
  2^{n-1} a & 1
\end{pmatrix}
\begin{pmatrix}
  p(1+z) & x \\
  0 & 1+z
\end{pmatrix}\\
&=&
\begin{pmatrix}
  p(1+z) & x \\
  2^{n-1}a p(1+z) & 2^{n-1}a x+1+z
\end{pmatrix}.
\end{IEEEeqnarray*}
Since 2 divides $x$, $2^{n-1}x=0 \mod 2^n$. Since $p \equiv 1 \mod 2$,
$$
2^{n-1}p(1+z) = 2^{n-1} p = 2^{n-1}
$$
in $\W/2^n \W$. Hence
$$
(I+2^{n-1} a g_p(\sigma_p)) \rho_n (\sigma_p) 
=
\begin{pmatrix}
  p(1+z) & x \\
  2^{n-1}a & 1+z
\end{pmatrix}.
$$
Clearly 
$$
(I+2^{n-1} a g_p(\tau_p)) \rho_n (\tau_p) 
=\begin{pmatrix}
1 & y \\
0 & 1
\end{pmatrix}
.
$$
In the remaining part of the proof, 
we will show that these two matrices 
are conjugate to the original matrices $\rho_n(\sigma_p)$ and $\rho_n(\tau_p)$ by a matrix that reduces to the identity mod 2. 
Let $c \in \W/2^n \W$, and compute  
\begin{IEEEeqnarray*}{rCl}
&& \bigg( I+
\begin{pmatrix}
0 & 0 \\
c & 0
\end{pmatrix}
\bigg)
\begin{pmatrix}
p(1+z) & x \\
0 & (1+z)
\end{pmatrix} 
\bigg( I-
\begin{pmatrix}
0 & 0 \\
c & 0
\end{pmatrix}
\bigg) \\
&=&
\begin{pmatrix}
  p(1+z) -cx & x \\ 
  c(p-1)(1+z) -c^2x & 1+z + cx 
\end{pmatrix}.
\end{IEEEeqnarray*}
Since $p-1$ is divisible by 2 and not by 4, 
$$
\frac{p-1}{2} \in \W^\times.
$$
If we let 
$$
c := 2^{n-2} \bigg( \frac{p-1}{2} \bigg)^{-1} \in \W
$$
we observe the following: First,
$$
c(p-1) = 2^{n-1}.
$$
Since 2 divides $z$, 
$$
c(p-1)(1+z) =2^{n-1}(1+z) = 2^{n-1}.   
$$
Note that $c^2x = 0, $ since this term is divisible by $2^{2(n-2)+1}$, and $ 2(n-2)+1\geq n$ for $n \geq 3.$
Also,
$$
p(1+z) - cx = p(1+z+cx-cx)- cx = p(1+z+cx) - cx(p+1) = p(1+z+cx),
$$
since $p+1$ is even and hence  
$
cx(p+1) 
$ is divisible by $2^n$.
In total, we have shown that 
$$
\bigg( I+
\begin{pmatrix}
0 & 0 \\
c & 0
\end{pmatrix}
\bigg)
\begin{pmatrix}
p(1+z) & x \\
0 & (1+z)
\end{pmatrix}
\bigg( I-
\begin{pmatrix}
0 & 0 \\
c & 0
\end{pmatrix}
\bigg) 
=
\begin{pmatrix}
  p(1+z+cx) & x \\ 
  2^{n-1}  & 1+z + cx 
\end{pmatrix}.
$$
Let $a_0 \in \W$ be a lift of $a\in \F_4$. A similar computation to the one above shows that 
\begin{IEEEeqnarray*}{rCl}
&& \bigg( I+
\begin{pmatrix}
0 & 0 \\
a_0 c & 0
\end{pmatrix}
\bigg)
\begin{pmatrix}
p(1+z-a_0cx) & x \\
0 & (1+z-a_0cx)
\end{pmatrix}
\bigg( I-
\begin{pmatrix}
0 & 0 \\
a_0 c & 0
\end{pmatrix}
\bigg) \\
&=&
\begin{pmatrix}
  p(1+z) & x \\ 
  a_0 c(p-1)  & 1+z 
\end{pmatrix} \\
&=& \begin{pmatrix}
  p(1+z) & x \\ 
    2^{n-1} a  & 1+z 
\end{pmatrix}.
\end{IEEEeqnarray*}
Finally,
$$
\bigg( I+
\begin{pmatrix}
0 & 0 \\
a_0c & 0
\end{pmatrix}
\bigg)
\begin{pmatrix}
1 & y \\
0 & 1
\end{pmatrix}
\bigg( I-
\begin{pmatrix}
0 & 0 \\
a_0c & 0
\end{pmatrix}
\bigg) 
=
\begin{pmatrix}
  1-a_0cy & y \\ 
  a_0c-a_0^2c^2y-a_0c  & 1+a_0cy 
\end{pmatrix}.
$$
Since $\rho_n$ is unramified mod 4, it follows that 4 divides $y$, and therefore
$$
cy=0.
$$
Hence the   matrix above equals 
$$\begin{pmatrix}
  1 & y \\ 
    & 1 
\end{pmatrix}.$$
Since $n \geq 3$, it follows that $2$ divides $c$, and hence the matrix
$$I+
\begin{pmatrix}
0 & 0 \\
a_0c & 0
\end{pmatrix}
$$
reduces to the identity mod 2. We conclude that $(I+2^{n-1}ag_p^{\text{unr}})\rho_n$ is the same deformation as
$$
\rho_n': \sigma_p \mapsto
\begin{pmatrix}
p(1+z-a_0cx) & x \\
0 & (1+z-a_0cx)
\end{pmatrix}, \quad 
\tau_p \mapsto
\begin{pmatrix}
  1 & y \\ 
    & 1 
\end{pmatrix},
$$
and clearly, $\rho_n' \in \mathcal{C}_p^{\operatorname{unr},\geq 3}$. This concludes the proof. Note that if $4$ divides $x$, then $(I+2^{n-1}ag_p^{\operatorname{unr}})\rho_n$ is the same deformation as $\rho_n$.
\end{proof}

\begin{corollary} \label{N_p}
The space $\mathcal{N}_p^{\operatorname{unr}}$ that preserves $\mathcal{C}_p^{\operatorname{unr}, \geq 3}$ is spanned by 
$f_p^{(1)},f_p^{(2)},f_p^{(3)}$ and $g_p^{\operatorname{unr}}$.
\end{corollary}

Let $\mathcal{C}_p^{\text{ram}}$ be the subset of $\mathcal{C}_p$ consisting of deformations that are ramified mod 4. 
Let $\mathcal{C}_p^{\text{ram}, \geq 3}$ be the subset of $\mathcal{C}_p^{\text{ram}}$ consisting of deformations to $\operatorname{GL}(2,\W/2^n\W)$ for $n\geq 3$ which are all deformations of the same mod 4 lift of the form
$$
\sigma_p \mapsto \begin{pmatrix}
p(1+z_0) & x_0 \\ 
0 & 1+z_0
\end{pmatrix},
\quad \tau_p \mapsto \begin{pmatrix}
1 & y_0 \\
0 & 1
\end{pmatrix},
$$
for some $x_0,y_0,z_0 \in 2\W/4\W$, where 4 does not divide $y_0$. In particular, 
$
\frac{y_0}{p-1}
$
defines an element of $\F_4^\times$.
Define $g_p^{\text{ram}}\in H^1(G_p,Ad(\overline{\rho}))$ as follows:
$$
g_p^{\text{ram}}: 
\sigma_p \mapsto 
\begin{pmatrix}
0 & 0 \\
1 & 0
\end{pmatrix}
, \quad \tau_p \mapsto 
\begin{pmatrix}
\frac{y_0}{p-1} & 0 \\
0 & \frac{y_0}{p-1}
\end{pmatrix}.
$$

\begin{proposition} \label{g^ram}
Let $\rho_n \in \mathcal{C}_p^{\operatorname{ram}, \geq 3}$, and let $a\in \F_4$.
Then 
$(I+2^{n-1}a g_p^{\operatorname{ram}})\rho_n \in \mathcal{C}_p^{\operatorname{ram}, \geq 3}$.
\end{proposition}
\begin{proof}
There are $x,y,z \in \W/2^n \W$ such that 
$$\rho_n: 
\sigma_p \mapsto \begin{pmatrix}
p(1+z) & x \\ 
0 & 1 +z
\end{pmatrix},
\quad \tau_p \mapsto \begin{pmatrix}
1 & y \\
0 & 1
\end{pmatrix},
$$
where $y \equiv y_0 \mod 4$. In particular, 
$$
2^{n-2} y_0 = 2^{n-2} y
$$
in $\W/2^n \W$, and hence
$$
2^{n-1} \frac{y_0}{p-1} 
= 2^{n-2}y_0 \frac{2}{p-1}
= 2^{n-2}y \frac{2}{p-1}
$$
Let 
$$
c=2^{n-2}\frac{2}{p-1}. 
$$
Then 
$$
2^{n-1} \frac{y_0}{p-1} = cy.
$$
Using this, together with $2^{n-1}y = 0$, we compute
\begin{IEEEeqnarray*}{rCl}
(I+2^{n-1} a g_p^{\text{ram}}(\tau_p)) \rho_n (\tau_p) 
&=&
\begin{pmatrix}
  1 - 2^{n-1}a\frac{y_0}{p-1} & 0 \\
  0 & 1+ 2^{n-1} a\frac{y_0}{p-1}
\end{pmatrix}
\begin{pmatrix}
  1 & y \\
  0 & 1
\end{pmatrix} \\
&=&
\begin{pmatrix}
  1-2^{n-1}a\frac{y_0}{p-1} & y \\
  0 & 1+2^{n-1}a\frac{y_0}{p-1}
\end{pmatrix} \\
&=& \begin{pmatrix}
  1-acy & y \\
  0 & 1+acy
\end{pmatrix}.
\end{IEEEeqnarray*}
At the same time,
\begin{IEEEeqnarray*}{rCl}
\bigg( I+
\begin{pmatrix}
0 & 0 \\
ac & 0
\end{pmatrix}
\bigg)
\begin{pmatrix}
1 & y \\
0 & 1
\end{pmatrix}
\bigg( I-
\begin{pmatrix}
0 & 0 \\
ac & 0
\end{pmatrix}
\bigg) 
&=&
\begin{pmatrix}
  1-acy & y \\ 
  ac-(ac)^2y-ac  & 1+acy 
\end{pmatrix}\\
&=&
\begin{pmatrix}
  1-acy & y \\ 
0  & 1+acy 
\end{pmatrix}
\end{IEEEeqnarray*}
Next, 
$$
(I+2^{n-1} a g_p^{\text{ram}}(\sigma_p)) \rho_n (\sigma_p) 
=
\begin{pmatrix}
  1  & 0 \\
  2^{n-1}a & 1 
\end{pmatrix}
\begin{pmatrix}
  p(1+z) & x \\
  0 & 1+z
\end{pmatrix}
=
\begin{pmatrix}
  p(1+z) & x \\
  2^{n-1}a & 1+z
\end{pmatrix} $$
while
$$
\bigg( I+
\begin{pmatrix}
0 & 0 \\
ac & 0
\end{pmatrix}
\bigg)
\begin{pmatrix}
p(1+z-acx) & x \\
0 & 1+z-acx
\end{pmatrix}
\bigg( I-
\begin{pmatrix}
0 & 0 \\
ac & 0
\end{pmatrix}
\bigg) 
=
\begin{pmatrix}
  p(1+z) & x \\ 
  2^{n-1}a  & 1+z
\end{pmatrix}.
$$
\end{proof}

\begin{corollary} \label{N_p^ram}
The space $\mathcal{N}_p^{\operatorname{ram}}$ 
that preserves $\mathcal{C}_p^{\operatorname{ram}, \geq 3}$ is
spanned by $f_p^{(1)},f_p^{(2)},f_p^{(3)}$ and $g_p^{\text{ram}}.$ 
\end{corollary}

In Lemma \ref{v,w}, we defined a pair of trivial primes $(v,w)$ at which the lift $\rho_2$ is ramified. At the prime $v$, we let $\mathcal{C}_v$ be the class of lifts $\pi_v$ defined above for any trivial prime $p$, and we let $\mathcal{N}_v$ be the space that preserves $\mathcal{C}_v^{\text{ram},\geq 3}$. At the trivial prime $w$, the definition of $(\mathcal{C}_w, \mathcal{N}_w)$ is similar to $(\mathcal{C}_v, \mathcal{N}_v)$, except that the matrices involved are lower triangular instead of upper triangular. We give the precise definitions in the remaining part of this section.

Let $w$ be a trivial prime as in Lemma \ref{v,w}. Let
$$
\pi_w : 
\sigma_w \mapsto 
\begin{pmatrix}
w^{-1}(1+z) & 0 \\
x & 1+z
\end{pmatrix}, \quad
\tau_w \mapsto 
\begin{pmatrix}
1 & 0 \\
y & 1
\end{pmatrix}
$$
where $x, y$ and $z$ in $2\W/2^n\W$ are free to vary. Since
$$
\begin{pmatrix}
w^{-1}(1+z) & 0 \\
x & 1+z
\end{pmatrix}
\begin{pmatrix}
1 & 0 \\
y & 1
\end{pmatrix}
\begin{pmatrix}
w^{-1}(1+z) & 0 \\
x & 1+z
\end{pmatrix}^{-1}
=
\begin{pmatrix}
1 & 0 \\
y w & 1
\end{pmatrix},
$$
$\pi_w$ defines a homomorphism on the tame quotient of the local Galois group at $w$. Again, $\pi_w$ is always liftable to characteristic zero. Let $\mathcal{C}_w$ denote the set of lifts of the form $\pi_w$. 
Let $\mathcal{C}_w^{\text{ram}}$ be the subset of $\mathcal{C}_w$ consisting of deformations that are ramified mod 4. 
Let $\mathcal{C}_w^{\text{ram}, \geq 3}$ be the subset of $\mathcal{C}_w^{\text{ram}}$ consisting of deformations to $GL(2,\W/2^n\W)$ for $n\geq 3$ which are all deformations of the same mod 4 lift of the form
$$
\sigma_w \mapsto \begin{pmatrix}
w^{-1}(1+z_0) &  \\ 
x_0 & 1+z_0
\end{pmatrix},
\quad \tau_p \mapsto \begin{pmatrix}
1 &  \\
y_0 & 1
\end{pmatrix},
$$
for some $x_0,y_0,z_0 \in 2\W/4\W$.
\begin{corollary} \label{corollary N_w}
The space $\mathcal{N}_w$ that preserves $\mathcal{C}_w^{\operatorname{ram},\geq 3}$
is spanned by the cohomology classes
$$
f_w^{(1)}: 
\sigma_w \mapsto 
\begin{pmatrix}
0 & 0 \\
1 & 0
\end{pmatrix}
, \quad \tau_w \mapsto 
\begin{pmatrix}
0 & 0 \\
0 & 0
\end{pmatrix};
\quad 
f_w^{(2)}: 
\sigma_w \mapsto 
\begin{pmatrix}
0 & 0 \\
0 & 0
\end{pmatrix}
, \quad \tau_w \mapsto 
\begin{pmatrix}
0 & 0 \\
1 & 0
\end{pmatrix}
$$
and 
$$
f_w^{(3)}: 
\sigma_w \mapsto 
\begin{pmatrix}
1 & 0 \\
0 & 1
\end{pmatrix}
, \quad \tau_w \mapsto 
\begin{pmatrix}
0 & 0 \\
0 & 0
\end{pmatrix};
\quad 
g_w^{\operatorname{ram}}: 
\sigma_w \mapsto 
\begin{pmatrix}
0 & 1 \\
0 & 0
\end{pmatrix}
, \quad \tau_w \mapsto 
\begin{pmatrix}
\frac{y_0}{w-1} & 0 \\
0 & \frac{y_0}{w-1}
\end{pmatrix}.
$$
We will also denote $g_w^{\operatorname{ram}}$ by $f_w^{(4)}$.
\end{corollary}
\begin{proof}
    It is easy to verify that $f_w^{(1)},f_w^{(2)}$ and $f_w^{(3)}$ preserve
    $\mathcal{C}_w$, so in particular they preserve $\mathcal{C}_w^{\operatorname{ram},\geq 3}$. The proof that  
    $g_w^{\operatorname{ram}}$ preserves  $\mathcal{C}_w^{\operatorname{ram},\geq 3}$ is similar to the proof 
    of Lemma \ref{g^ram}.
\end{proof}

\section{Local liftability mod 4 at the ramified primes}
Let $S=\{\infty, 2, 7, l, v, w\}$. Ultimately, we will show that $$\rho_2|_{G_q} = (I+2(f^{(v)}+f^{(w)}))\xi_2|_{G_q} \in \mathcal{C}_q$$ for all $q \in S$. First, we deal with $\infty, 2, 7,$ and $l$.
We have already seen that $\xi_2 \in \mathcal{C}_q$ for every $q\in \{\infty, 2, 7, l\}$. Hence, it suffices to show that 
$$(f^{(v)}+f^{(w)})|_{G_q} \in \mathcal{N}_q$$ for every $q \in \{\infty, 2, 7, l\}$. 
By Lemma \ref{H^1(G_7)}, 
$$
(f^{(v)}+f^{(w)})|_{G_7} \in
H^1(G_7, 
\begin{pmatrix}
& * \\
* &
\end{pmatrix}
) = 0,
$$
and by Lemma \ref{H^1(G_l)}, 
$$
(f^{(v)}+f^{(w)})|_{G_l} \in
H^1(G_l, 
\begin{pmatrix}
& * \\
* &
\end{pmatrix}
) = 0.
$$
Next, since the fields $K^{(v)}$ and $K^{(w)}$ are totally real, 
$(f^{(v)}+f^{(w)})|_{G_\infty} =0$.
By Lemma \ref{N_2}, 
$$
(f^{(v)}+f^{(w)})|_{G_2} \in
H^1(G_2, 
\begin{pmatrix}
& * \\
* &
\end{pmatrix}) \subseteq \mathcal{N}_2.
$$
Next, we consider the primes $v$ and $w$. Recall that we have chosen $v$ and $w$ to remain prime in $\Q(\sqrt{l})$. Since $\chi_l: G_\Q \to \F_2$ is an additive character, it follows that
$$
\chi_l(\sigma_v)=1, \quad \chi_l(\sigma_w)=1.
$$
Also, $v$ and $w$ are trivial primes, and hence $\psi|_{G_v}$ and $\psi|_{G_w}$ are trivial. We find 
$$
\xi_2|_{G_v}: \sigma_v \mapsto
\begin{pmatrix}
3 & \\
& 1
\end{pmatrix} =
\begin{pmatrix}
v & \\
& 1
\end{pmatrix},
\quad \tau_v \mapsto 
I$$
since $v \equiv 3 \mod 4$. We conclude that $\xi_2 \in \mathcal{C}_v$. Next, 
$$
\xi_2|_{G_w}: \sigma_w \mapsto
\begin{pmatrix}
3 & \\
& 1
\end{pmatrix} =
\begin{pmatrix}
w^{-1} & \\
& 1
\end{pmatrix},
\quad \tau_w \mapsto 
I.$$
Here, we use that $w \equiv 3 \mod 4$, and hence $w^{-1} \in \W$ also satisfies $w^{-1} \equiv 3 \mod 4$.
We conclude that $\xi_2|_{G_w} \in \mathcal{C}_w$. 

\emph{Remark:} The reason for introducing ramification at $l$ by twisting the upper left corner by $\chi_l$ is precisely to force $\xi_2$ to be of the desired form at $v$ and $w$.

Clearly, $f^{(v)}|_{G_v} $ preserves $\mathcal{C}_v$ 
and $f^{(w)}|_{G_w}$ preserves $\mathcal{C}_w.$
Recall that we started by choosing $v$, and then we chose $w$ such that $f^{(v)}(\sigma_w)=0$. This is equivalent to $f^{(v)}|_{G_w} =0$. 
Hence, it only remains to show that $f^{(w)}|_{G_v} = 0$. We use the global reciprocity law.

\begin{lemma} \label{localpairing}
Let $p$ be a trivial prime. The unramified cohomology classes in 
$H^1(G_p, \F_4 )$ and $H^1(G_p, (\F_4)^* )$ are exact annihilators of each other under the local pairing. If $f \in H^1(G_p, \F_4)$ is unramified and $\varphi \in H^1(G_p, (\F_4)^*)$ is ramified, then 
$$
f \neq 0 \implies f \cup \varphi \neq 0.
$$
If $f \in H^1(G_p, \F_4)$ is ramified and $\varphi \in H^1(G_p, (\F_4)^*)$ is unramified, then 
$$
\psi \neq 0 \implies f \cup \varphi \neq 0.
$$
\end{lemma}
\begin{proof}
    See \cite{HR} Fact 37, p. 934. 
\end{proof}

\begin{lemma} \label{globalreciprocity}
$$
f^{(v)}|_{G_w} = 0
\iff 
f^{(w)}|_{G_v} = 0.
$$
\end{lemma}
\begin{proof}
Recall that $f^{(v)} \in 
H^1(G_\Q, \F_4(\psi^{-1}))$.
By Lemma \ref{character},
$$
\F_4(\psi^{-1})^* \simeq \F_4(\psi),
$$
which implies that
$$
H^1(G_\Q, \F_4(\psi^{-1})^*)
\simeq
H^1(G_\Q, \F_4(\psi)).
$$
It follows that we may regard $f^{(w)}$ as an element of 
$H^1(G_\Q, \F_4(\psi^{-1})^*)$.
By the global reciprocity law,
$$
\sum_p \text{inv}_p ( 
f^{(v)}|_{G_p} \cup f^{(w)}|_{G_p} )= 0.
$$
Let $p$ be a prime different from $v$ and $w$, let $f$ and $M$ denote $f^{(v)}$ and $\F_4(\psi^{-1})$ or $f^{(w)}$ and $\F_4(\psi^{})$, respectively, 
and consider the commutative diagram
\[
\begin{tikzcd}
H^1(G_p/I_p, M^{I_p}) \arrow[r, "\text{inf}"]
& H^1(G_p, M) \arrow[r] \arrow[d, "\text{res}" ]
 & H^1(I_p,M) \arrow[d, "\text{res}"]\\
 & H^1(G_p\cap \ker\overline{\rho},M) \arrow[r] & H^1(I_p \cap \ker\overline{\rho},M) 
\end{tikzcd}
\]\
We have injective maps
\[
\begin{tikzcd}
I_p/I_p \cap \ker\overline{\rho} \arrow[r, hook] 
& G_p/ G_p \cap \ker\overline{\rho} \arrow[r, hook] 
& G_\Q / \ker\overline{\rho},
\end{tikzcd}
\]
so the order of these quotients all divide 3. Since $\#M =4$, it follows that the vertical restriction maps are injective (see \cite{koch} Theorem 3.15, page 34). 
Since $K^{(v)}/K$ and $K^{(w)}/K$ are unramified at $p$, 
it follows that 
$$
\big( f|_{\ker\overline{\rho}} \big)|_{I_p \cap \ker\overline{\rho}}
= f|_{I_p \cap \ker\overline{\rho}} = 0.
$$
Since the right-most vertical map in the diagram above is injective, it follows that 
$$
f|_{I_p} =0.
$$
By exactness of the top row, $f$ inflates from $H^1(G_p/I_p,M^{I_p})$. 
We conclude that 
$$f^{(v)}|_{G_p} \in H^1(G_p/I_p,\F_4(\psi^{-1})^{I_p}), \quad f^{(w)}|_{G_p} \in 
H^1(G_p/I_p,\F_4(\psi)^{I_p})$$ 
and hence
$$
f^{(v)}|_{G_p} \cup f^{(w)}|_{G_p} \in H^2(G_p/I_p, \mu) = 0.
$$
Here, $\mu$ is the roots of unity. The cohomological triviality follows from the fact that $\mu$ is a torsion group for which $\hat{\Z}$ has cohomological dimension one. Hence, the global reciprocity law implies that
$$
\text{inv}_v( 
f^{(v)}|_{G_v} \cup f^{(w)}|_{G_v})
= 
-\text{inv}_w(
f^{(v)}|_{G_w} \cup f^{(w)}|_{G_w}
).
$$
If $p$ is a trivial prime, then the $G_p$ action on $Ad(\overline{\rho})$ is trivial, and hence
$$
H^1(G_p, 
\begin{pmatrix}
& * \\
&
\end{pmatrix}
)
=
H^1(G_p, \F_4).
$$
Note that $f^{(v)}|_{G_w}\in H_{\operatorname{unr}}^1(G_w, \F_4)$.
Assume $f^{(v)}|_{G_w} = 0$. 
Then
$$
f^{(v)}|_{G_w} \cup f^{(w)}|_{G_w} = 0,
$$
and by global reciprocity, 
$$
f^{(v)}|_{G_v} \cup f^{(w)}|_{G_v} =0.
$$
Since $f^{(v)}|_{G_v} \in H^1(G_v, \F_4)$ is ramified and 
$f^{(w)}|_{G_v} \in H^1(G_v, \F_4)$ is unramified, it follows from Lemma \ref{localpairing} that 
$$
f^{(w)}|_{G_v} =0.
$$
The other implication follows similarly.
\end{proof}

\section{The global setting}
Let $S=\{\infty,2,7,l,v,w\}$. Recall that $G_S = \Gal(\Q_S/\Q)$ where $\Q_S$ is the maximal extension of $\Q$ unramified outside $S$. For any $G_S$-module $M$ over $\F_4$, let $\Sha^1_S(M)$ be the kernel of the map 
$$
H^1(G_S, M) \to \bigoplus_{q\in S} H^1(G_q, M).
$$
\begin{lemma} \label{trivialsha}
    Let $M$ be a $\F_4[G_S]$-submodule of $Ad(\overline{\rho})$ or $Ad(\overline{\rho})^*$. Then
$$\Sha^1_S(M) = 0.$$
\end{lemma}

\begin{proof}
 Let $f \in \Sha_S^1(M)$. Suppose $f\neq 0$. 
 Note that $M$ is naturally a $\F_4[G_S/\ker\overline{\rho}]$-module. 
Since the order of $G_S / \ker\overline{\rho}$ is prime to the order of $M$, it follows from \cite{koch} Theorem 3.15 that
 $$
 H^1( G_S/ \ker\overline{\rho}, M ) = 0.
 $$
 We conclude that $f|_{\ker\overline{\rho}} \neq 0$. Hence $f$ cuts out a nontrivial extension $K_f$ of $K$ unramified outside $S$. Since $f|_{G_q}=0$ for all $q\in S$, all the primes in $S$ split completely in $K_f$. In particular, the extension $K_f/K$ must be unramified everywhere. But since the class group of $K$ is trivial, there are no unramified extensions of $K$.
\end{proof}

Recall that $\rho_2 \in \mathcal{C}_q$ for all $q\in S$. Hence there are no local obstructions to deforming $\rho_2$ to $\W/8\W$. Since $S$ contains $\infty$ and $2$, we know that 
$
\Sha^2_S(Ad(\overline{\rho})) 
$
is dual to $\Sha^1_S(Ad(\overline{\rho})^*)$, which is trivial by Lemma \ref{trivialsha}. 
By Proposition \ref{local-global lift}, we may deform $\rho_2$ to a representation
$$\rho_3: G_S \to GL(2, \W/8\W).
$$
Next, we recall Wiles' formula (see Proposition 1.6 in \cite{Wiles}). Let $T$ be a finite set of primes containing $S=\{\infty,2,7,l,v,w \}$. Let $M$ be a Galois stable subspace of $Ad(\overline{\rho})$ or $Ad(\overline{\rho})^*$. For each $q\in S,$ let $\mathcal{L}_q$ be a subspace of $H^1(G_q, M)$ with annihilator $\mathcal{L}_q^\perp \subseteq H^1(G_q, M^*)$ under the local pairing. Then 
\begin{IEEEeqnarray*}{ll}
 & \dim H^1_\mathcal{L} (G_T, M) - \dim H^1_{\mathcal{L}^\perp} (G_T, M^*) \\
= \quad & 
\dim H^0 (G_T, M) - \dim H^0 (G_T, M^*) 
+
\sum_{q\in T} \dim \mathcal{L}_q - \dim H^0(G_q, M).  
\end{IEEEeqnarray*}

\begin{proposition} \label{globalsetting}
The global setting has the property that   
$$
\dim H^1_{\mathcal{N}^\perp}(G_S, Ad(\overline{\rho})^*)
=
\dim H^1_{\mathcal{N}}(G_S, Ad(\overline{\rho})) 
+1.
$$
\end{proposition}

\begin{proof}
By Lemma \ref{N_7},
$$
\dim \mathcal{N}_7 - \dim H^0(G_7, Ad(\overline{\rho})) = 0,
$$
and by Lemma \ref{N_l}, 
$$
\dim \mathcal{N}_l - \dim H^0(G_l, Ad(\overline{\rho})) = 0.
$$
Recall that 
$\mathcal{N}_v = \mathcal{N}_v^{\operatorname{ram}}$
and $\mathcal{N}_w = \mathcal{N}_w^{\operatorname{ram}}$. 
By Corollary \ref{N_p^ram}, it holds that 
$$
\dim \mathcal{N}_v - \dim H^0(G_v, Ad(\overline{\rho})) = 0,
\quad
\dim \mathcal{N}_w - \dim H^0(G_w, Ad(\overline{\rho})) = 0.
$$
By Proposition \ref{R_2} and the remarks at the end of Section
\ref{2 and infinity},
$$
\dim \mathcal{N}_\infty  - \dim H^0(G_\infty, Ad(\overline{\rho})) = -4, \quad 
\dim \mathcal{N}_2 - \dim H^0(G_2, Ad(\overline{\rho})) = 5 - 2 = 3.
$$
Clearly, 
$$
\dim H^0 (G_S, Ad(\overline{\rho})) = \dim H^0 (G_S, Ad(\overline{\rho})^*) =2.
$$
By Wiles' formula (Lemma \ref{Wiles}),
\begin{IEEEeqnarray*}{ll}
 & \dim H^1_\mathcal{N} (G_S, Ad(\overline{\rho})) 
 - \dim H^1_{\mathcal{N}^\perp} (G_S, Ad(\overline{\rho})^*) \\
= \quad & 
\dim H^0 (G_S, Ad(\overline{\rho})) - \dim H^0 (G_S, Ad(\overline{\rho})^*) 
+
\sum_{q\in S} \dim \mathcal{N}_q - \dim H^0(G_q, Ad(\overline{\rho})).  
\end{IEEEeqnarray*}

It follows that 
$$
\dim H^1_{\mathcal{N}^\perp}(G_S, Ad(\overline{\rho})^*) 
=
\dim H^1_{\mathcal{N}}(G_S, Ad(\overline{\rho})) + 1.$$
\end{proof}

\section{The Selmer group 
\texorpdfstring{$H^1_{\mathcal{N}}(G_S, Ad(\overline{\rho}))$}{Selmer}
 is of rank two}
We start by recalling the definition of the spaces $\mathcal{N}_v$ and $\mathcal{N}_w$.

\begin{lemma} \label{N_v}
The space $\mathcal{N}_v$ is spanned by the cohomology classes
$$
f_v^{(1)}: 
\sigma_v \mapsto 
\begin{pmatrix}
0 & 1 \\
0 & 0
\end{pmatrix}
, \quad \tau_v \mapsto 
\begin{pmatrix}
0 & 0 \\
0 & 0
\end{pmatrix},
\quad 
f_v^{(2)}: 
\sigma_v \mapsto 
\begin{pmatrix}
0 & 0 \\
0 & 0
\end{pmatrix}
, \quad \tau_v \mapsto 
\begin{pmatrix}
0 & 1 \\
0 & 0
\end{pmatrix}
$$
and 
$$
f_v^{(3)}: 
\sigma_v \mapsto 
\begin{pmatrix}
1 & 0 \\
0 & 1
\end{pmatrix}
, \quad \tau_v \mapsto 
\begin{pmatrix}
0 & 0 \\
0 & 0
\end{pmatrix},
\quad 
g_v^{\text{\upshape{ram}}}: 
\sigma_v \mapsto 
\begin{pmatrix}
0 & 0 \\
1 & 0
\end{pmatrix}
, \quad \tau_v \mapsto 
\begin{pmatrix}
1 & 0 \\
0 & 1
\end{pmatrix}.
$$
We will also denote $g_v^{\text{\upshape{ram}}}$ by $f_v^{(4)}$.
\end{lemma}
\begin{proof}
Since $\rho_2$ is ramified at $v$, the space $\mathcal{N}_v$ is defined as in Corollary \ref{N_p^ram}. We may normalize $f^{(v)}$ such that 
$$
f^{(v)}(\tau_v) =\begin{pmatrix}
    & 1 \\
    &
\end{pmatrix}.
$$
Then 
$$
\rho_2|_{G_v}: 
\sigma_v \mapsto
\begin{pmatrix}
  v  & 2x \\
  & 1
\end{pmatrix}, \quad 
\tau_v \mapsto
\begin{pmatrix}
    1 & 2 \\
    & 1
\end{pmatrix},
$$
where $x \in \F_4$ depends on $f^{(v)}(\sigma_v)$. It follows that 
the constant $y_0$ used to define $g_v^{\text{ram}}$ equals 2, and
$$
\frac{2}{v-1} \equiv 1 \mod 2. 
$$
\end{proof}

\begin{lemma} \label{N_w}
The space $\mathcal{N}_w$ is spanned by the cohomology classes
$$
f_w^{(1)}: 
\sigma_w \mapsto 
\begin{pmatrix}
0 & 0 \\
1 & 0
\end{pmatrix}
, \quad \tau_w \mapsto 
\begin{pmatrix}
0 & 0 \\
0 & 0
\end{pmatrix},
\quad 
f_w^{(2)}: 
\sigma_w \mapsto 
\begin{pmatrix}
0 & 0 \\
0 & 0
\end{pmatrix}
, \quad \tau_w \mapsto 
\begin{pmatrix}
0 & 0 \\
1 & 0
\end{pmatrix}
$$
and 
$$
f_w^{(3)}: 
\sigma_w \mapsto 
\begin{pmatrix}
1 & 0 \\
0 & 1
\end{pmatrix}
, \quad \tau_w \mapsto 
\begin{pmatrix}
0 & 0 \\
0 & 0
\end{pmatrix},
\quad 
g_w^{\text{\upshape{ram}}}: 
\sigma_w \mapsto 
\begin{pmatrix}
0 & 1 \\
0 & 0
\end{pmatrix}
, \quad \tau_w \mapsto 
\begin{pmatrix}
1 & 0 \\
0 & 1
\end{pmatrix}.
$$
We will also denote $g_w^{\text{\upshape{ram}}}$ by $f_w^{(4)}$.
\end{lemma}

\begin{proof}
Since $\rho_2$ is ramified at $w$, the space $\mathcal{N}_w$ is defined as in Corollary \ref{corollary N_w}. We may normalize $f^{(w)}$ such that 
$$
f^{(w)}(\tau_w) =\begin{pmatrix}
    &  \\
   1 &
\end{pmatrix}.
$$
Then 
$$
\rho_2|_{G_w}: 
\sigma_v \mapsto
\begin{pmatrix}
  w^{-1}  &  \\
  2x & 1
\end{pmatrix}, \quad 
\tau_v \mapsto
\begin{pmatrix}
    1 &  \\
   2 & 1
\end{pmatrix},
$$
where $x \in \F_4$ depends on $f^{(w)}(\sigma_w)$. It follows that 
the constant $y_0$ used to define $g_w^{\text{ram}}$ equals 2, and 
$$
\frac{2}{v-1} \equiv 1 \mod 2. 
$$
\end{proof}

\begin{lemma}
Let $\chi_2$ be the quadratic character corresponding to $\mathbb{Q}(\sqrt{2})/\Q$, regarded as 
as additive character  $$\chi_2: G_S \to \F_2.$$
Let $a,b \in \F_4$ be constants. 
The homomorphism 
\begin{IEEEeqnarray*}{rCl} 
\chi_2
\begin{pmatrix}
 a & \\
  & b
\end{pmatrix} &:&
G_S \to 
\begin{pmatrix}
* & \\
& * 
\end{pmatrix} \\
g &\mapsto & 
\chi_2(g)
\begin{pmatrix}
 a & \\
  & b
\end{pmatrix}
\end{IEEEeqnarray*}
is cohomology class in $H^1(G_S, \begin{pmatrix}
* & \\
& * 
\end{pmatrix})$ whose restriction to $G_2$ lives in $\mathcal{N}_2$, and 
$$
\chi_2
\begin{pmatrix}
 a & \\
  & b
\end{pmatrix} \in H^1_{\mathcal{N}}(G_S, Ad(\overline{\rho}))
$$
if and only if $a=b.$
\end{lemma}
\begin{proof}
By definition, the character $\chi_2$ is the projection 
$G_\Q \to \Gal(\Q(\sqrt{2})/\Q)\simeq \F_2.$ Hence, $\chi_2$ factors through the projection $\Gal(\Q(\zeta_{2^\infty})/\Q) \to \Gal(\Q(\sqrt{2})/\Q),$
so we may regard $\chi_2$ as a character on $\Gal(\Q(\zeta_{2^\infty})/\Q).$ Let $c$ denote complex conjugation 
in $\Gal(\Q(\zeta_{2^\infty})/\Q).$ The subgroup $\langle c \rangle$ cuts out the totally real subfield 
$\Q(\zeta_{2^\infty})^+$, and 
$$\Gal(\Q(\zeta_{2^\infty})/\Q) \simeq \Gal(\Q(\zeta_{2^\infty})^+/\Q) \times \langle c \rangle
\simeq \Z_2 \times \Z/2\Z.
$$
Since $\Q(\sqrt{2})$ is totally real, 
$\chi_2(c)=0.$ Hence, $\chi_2$ vanishes on the torsion part 
$\langle c \rangle \simeq \Z/2\Z$. Since $\Q(\zeta_{2^\infty})/\Q$ is totally ramified at 2, 
the global Galois group $\Gal(\Q(\zeta_{2^\infty})/\Q)$ is equal to the decomposition group at 2, which is isomorphic to
$\Gal(\Q_2(\zeta_{2^\infty})/\Q_2)$. Recall that  
$$H^1(G_2, 
\begin{pmatrix}
* & \\
& * 
\end{pmatrix}
) = \text{Hom}( \Gal(\Q_2^\text{ab}/\Q_2), \F_4^2 ).$$ 
By Lemma \ref{N_2}, the homomorphisms in this space that live in $\mathcal{N}_2$ are precisely those whose restriction to 
$\Gal(\Q_2(\zeta_{2^\infty})/\Q_2)$ vanish on the torsion part. Regarding 
$\chi_2$ as an additive character with values in $\F_2$, we conclude that for any $a,b \in \F_4,$ there is a global cohomology class
\begin{IEEEeqnarray*}{rCl} 
\chi_2
\begin{pmatrix}
 a & \\
  & b
\end{pmatrix} &:&
G_S \to 
\begin{pmatrix}
* & \\
& * 
\end{pmatrix} \\
g &\mapsto & 
\chi_2(g)
\begin{pmatrix}
 a & \\
  & b
\end{pmatrix}
\end{IEEEeqnarray*}
in $H^1(G_S, \begin{pmatrix}
* & \\
& * 
\end{pmatrix}) = \operatorname{Hom}(G_S, 
\begin{pmatrix}
* & \\
& * 
\end{pmatrix}
)$ whose restriction to $G_2$ lives in $\mathcal{N}_2$. We will show that 
$$
\chi_2
\begin{pmatrix}
 a & \\
  & b
\end{pmatrix} \in H^1_{\mathcal{N}}(G_S, Ad(\overline{\rho}))
$$ precisely when $a=b.$ We notice that
$$
\chi_2
\begin{pmatrix}
 a & \\
  & b
\end{pmatrix}|_{G_v}:
\sigma_v \mapsto 
\chi_2(\sigma_v) \begin{pmatrix}
 a & \\
  & b
\end{pmatrix},
\quad 
\tau_v \mapsto 0
$$
equals $ a \chi_2(\sigma_v) f_v^{(3)} \in \mathcal{N}_v$ if $a=b$ and is not in $\mathcal{N}_v$ if $a\neq b.$ The homomorphism 
$$
\chi_2
\begin{pmatrix}
 a & \\
  & a
\end{pmatrix}
$$ is clearly also in $\mathcal{N}_w$ when restricted to $G_w$, and hence defines an element in the Selmer group. 
\end{proof}

\begin{lemma} \label{chi_v}
Let $\chi_v$ denote the quadratic character corresponding to $\mathbb{Q}(\sqrt{v})/\Q.$ For any 
$a, b\in \F_4$, where $(a,b)\neq (0,0)$, the homomorphism 
$$
\chi_v
\begin{pmatrix}
  a &  \\
  & b
\end{pmatrix}
$$
is not in $\mathcal{N}_v$ when restricted to $G_v$. In particular, $\chi_v$
can never give rise to an element in the Selmer group $H^1_{\mathcal{N}}(G_S, Ad(\overline{\rho}))$. 
\end{lemma}

\begin{proof}
It suffices to show that 
$$
\chi_v
\begin{pmatrix}
  a &  \\
  & b
\end{pmatrix}
$$
is not in $\mathcal{N}_v$ when restricted to $G_v$ for any $a,b \in \F_4$ not both equal to zero.
By Lemma \ref{N_v}, the space $\mathcal{N}_v$ is spanned by four cohomology classes $f_v^{(1)},f_v^{(2)},f_v^{(3)}$ and $g_v^\text{ram}$. Let $f_v^{(4)}$ denote $g_v^\text{ram}$. Suppose
$$
\chi_v
\begin{pmatrix}
  a &  \\
  & b
\end{pmatrix}|_{G_v} =\sum_{i=1}^4 \lambda_i f_v^{(i)}
$$
Then 
$$
\chi_v (\sigma_v)
\begin{pmatrix}
  a &  \\
  & b
\end{pmatrix} =
\begin{pmatrix}
  \lambda_3 & \lambda_1  \\
  \lambda_4 & \lambda_3
\end{pmatrix}.
$$
Since $\chi_v (\sigma_v)=0,$ it follows that $\lambda_1, \lambda_3, \lambda_4=0.$ 
Next, 
$$
\chi_v (\tau_v)
\begin{pmatrix}
  a &  \\
  & b
\end{pmatrix} =
\begin{pmatrix}
  \lambda_4  & \lambda_2  \\
   & \lambda_4 
\end{pmatrix}.
$$
It follows that also $\lambda_2=0$. We may therefore conclude that 
$$
\chi_v
\begin{pmatrix}
  a &  \\
  & b
\end{pmatrix}|_{G_v} \notin \mathcal{N}_v.
$$
\end{proof}

\begin{lemma} \label{chi_w}
Let $\chi_w$ denote the quadratic character corresponding to $\mathbb{Q}(\sqrt{w})/\Q$. For any 
$a, b\in \F_4$, $(a,b)\neq (0,0)$, the homomorphism 
$$
\chi_w
\begin{pmatrix}
  a &  \\
  & b
\end{pmatrix}
$$
is not in $\mathcal{N}_w$ when restricted to $G_w$. In particular, $\chi_w$
can never give rise to an element in the Selmer group $H^1_{\mathcal{N}}(G_S, Ad(\overline{\rho}))$. 
\end{lemma}

\begin{proof} The proof is similar to the proof of Lemma \ref{chi_v} 
\end{proof}

\begin{lemma}
The cohomology class $f^{(v)}$ is in $H^1_{\mathcal{N}}(G_S, Ad(\overline{\rho}))$.
\end{lemma}

\begin{proof} We refer to Proposition \ref{globalclasses} for the definition of $f^{(v)}$. Since $K^{(v)}$ is totally real, $f^{(v)}|_{G_\infty} \in \mathcal{N}_\infty.$ 
By Lemma \ref{N_7},
$\mathcal{N}_7 =H^1_{\operatorname{unr}}(G_7, Ad(\overline{\rho})),$
and by Lemma \ref{N_l},
$\mathcal{N}_l =H^1_{\text{unr}}(G_l, Ad(\overline{\rho})).$
Since $K^{(v)}$ is unramified at primes above 7 and $l$, it follows that $f^{(v)}|_{G_7} \in \mathcal{N}_7$ and
$f^{(v)}|_{G_l} \in \mathcal{N}_l$. 
Since $f^{(v)} \in H^1(G_S, 
\begin{pmatrix}
    & *\\
    &
\end{pmatrix}
),$ and since 
$
H^1(G_v, 
\begin{pmatrix}
    & *\\
    &
\end{pmatrix}
)
\subseteq 
\mathcal{N}_v
$ by Lemma \ref{N_v}, it follows that $f^{(v)}|_{G_v} \in \mathcal{N}_v$. 
By Lemma \ref{N_2}, $f^{(v)}|_{G_2} \in \mathcal{N}_2$. Since $f^{(v)}|_{G_w}=0,$ we conclude that $f^{(v)} \in H^1_{\mathcal{N}}(G_S, Ad(\overline{\rho}))$. 
\end{proof}

\begin{lemma} \label{f^w}
The cohomology class $f^{(w)} \not \in H^1_{\mathcal{N}}(G_S, Ad(\overline{\rho}))$.
\end{lemma}
\begin{proof}
First, note that 
    $$
f^{(w)}|_{G_2} \in H^1(G_2,
\begin{pmatrix}
    & \\
   * &
\end{pmatrix}
).
$$
By Lemma \ref{N_2}, 
$$
\mathcal{N}_2 \cap 
H^1(G_2,
\begin{pmatrix}
    & \\
   * &
\end{pmatrix}
) = 0. 
$$
Since any Frobenius automorphism at 2 in $K^{(w)}$ has order at least 3, 
$
f^{(w)}|_{G_2} \neq 0.
    $ 
We conclude that 
    $
f^{(w)}|_{G_2} \not \in \mathcal{N}_2
    $
    and hence $f^{(w)} \not \in H^1_{\mathcal{N}}(G_S, Ad(\overline{\rho}))$.
\end{proof}

We recall a formula from class field theory that computes the rank of the maximal $p$-elementary abelian extension of a number field $F$ unramified outside a finite set of primes $T$ in $F$ (see \cite{koch} Theorem 11.8, p. 120). For any field $E$, let $\delta(E)$ be 1 if $E$ contains the $p$-th roots of unity and 0 otherwise. Let $\chi(\mathfrak{p})$ denote the characteristic of the residue field associated to the prime $\mathfrak{p}$ in $F$. Let $r$ denote the number of infinite places in $F$. Let $G_T$ be the Galois group of the maximal extension of $F$ unramified outside $T$.
Finally, let 
$$
V_T = \{ a \in F^\times| (a) = \mathfrak{a}^p \text{ for some fractional ideal $\mathfrak{a}$, and } \forall \mathfrak{p} \in T: a \in F_\mathfrak{p}^p  \},
$$
where $(a)$ denotes the fractional ideal generated by $a$. Then 
\begin{IEEEeqnarray*}{ll}
 & \dim_{\F_p} H^1(G_T, \Z / p\Z) \\
= & \sum_{\mathfrak{p}\in T: \chi(\mathfrak{p}) = p} (F_\mathfrak{p}: \Q_p) - \delta(F) - r + 1 
+ \sum_{\mathfrak{p}\in T} \delta(F_\mathfrak{p})
+ \dim_{\F_p} (V_{T}/F^{\times p})^*,
\end{IEEEeqnarray*}
where $(V_{T}/F^{\times p})^* = \text{Hom}(V_{T}/F^{\times p}, \F_p)$.
Earlier, we used a tame version of this formula from \cite{gras} to construct the fields $K^{(v)}$ and $K^{(w)}$.

\begin{lemma} \label{rayclass}
Let $S_0$ denote the primes in $K$ above 2, $v$ and $w$.
Let $M$ denote the maximal 2-elementary abelian extension of $K$ unramified outside 2, $v$ and $w$. Then $\Gal(M/K)$ is a vector space over $\F_2$ whose dimension equals
$$
7 + \dim_{\F_2} (V_{S_0}/K^{\times 2})^*.
$$
The prime $v=167$ has the following properties:
\begin{enumerate}
    \item $v$ is trivial
    \item if $\mathfrak{v}$ divides $v$ in $K$, then the Frobenius automorphism 
$\sigma_\mathfrak{v} = (1,0) \in \F_2 \oplus U_2$, where $\F_2 \oplus U_2 = \Gal(K(\sqrt{U_K})/K)$ as a $\Gal(K/\Q)$-module.
\item $v$ remains prime in $\Q(\sqrt{l})$.
\end{enumerate}
The prime $w=379$ has the following properties:
\begin{enumerate}
    \item $w$ is trivial
    \item if $\mathfrak{w}$ divides $w$ in $K$, then $\sigma_\mathfrak{w} = (1,0) \in \F_2 \oplus U_2$.
\item $w$ remains prime in $\Q(\sqrt{l})$
\item $w$ splits completely in $K^{(167)}$.
\end{enumerate}
If $S_0$ denotes the primes in $K$ above 2, 167 and 379, then 
$$
 (V_{S_0}/K^{\times 2})^* =0.
$$
\end{lemma}
\begin{proof}
The computation of the dimension of the maximal 2-elementary abelian extension of $K$ unramified outside $S_0$ follows immediately from the formula above. The primes $v=167$ and $w=379$ have been found by a search in Magma. The set of primes satisfying the properties listed for $v$ is clearly a Chebotarev set, so there are infinitely many choices for $v$. 
Given a fixed choice of $v$, the set of primes satisfying the properties listed for $w$ is also  a Chebotarev set. 
This shows that there are infinitely many pairs $(v,w)$ satisfying the listed properties. 
The prime $v=167$ is the first one with all the properties required of $v$. 
Having chosen $v=167$, the prime $w=379$ is the first prime one can choose with all the properties required of $w$.

If $S_0$ denotes the primes in $K$ above 2, 167 and 379, a ray class computation in Magma above the field $K$ shows that he maximal 2-elementary abelian extension of $K$ unramified outside $S_0$ is 7-dimensional as a vector space over $\F_2$, which implies that
$
 (V_{S_0}/K^{\times 2})^* =0.
$
\end{proof}
From now on, we assume that $v$ and $w$ have been chosen so that
$$(V_{S_0}/K^{\times 2})^* =0.$$

\begin{proposition}  \label{Selmer_S}
Let $S= \{\infty, 2, 7, l, v, w \}.$ 
The Selmer group $H^1_{\mathcal{N}}(G_S, Ad(\overline{\rho}))$ has rank 2, i.e.
$$
\dim_{\F_4} H^1_{\mathcal{N}}(G_S, Ad(\overline{\rho})) = 2,  
$$ 
and equals the span of the cohomology classes $f^{(v)}$ and  
$
\chi_2 I
$.
\end{proposition}

\begin{proof} 
Suppose $f \in H^1_{\mathcal{N}}(G_S, Ad(\overline{\rho}))$ and that $f\neq 0$. We will show that 
$f$ is in the linear span of $f^{(v)}$ and $\chi_2 I$.
The cohomology class $f$ cuts out a nontrivial extension $K_f$ of $K$ (the fixed field of $\ker f|_K$) which is contained in the maximal 2-elementary abelian extension $M$ of $K$ unramified outside $2, v$ and $w$, and
$$
\Gal(K_f/K) \simeq 
\text{Im}(f|_K).
$$
The fields 
$K(\sqrt{2}),$ $K(\sqrt{v}),$ $K(\sqrt{w}),$ $K^{(v)}$ and $K^{(w)}$ are all contained in $M$, and 
each field is linearly disjoint over $K$ with the compositum of the others.
Since $\Gal(M/K)$ is 7-dimensional, it follows that 
$M$ equals the compositum of these fields,
and that $\Gal(M/K)$ is isomorphic to
$$ 
\Gal(K(\sqrt{2})/K) \oplus
\Gal(K(\sqrt{v})/K) \oplus
\Gal(K(\sqrt{w})/K) \oplus
\Gal(K^{(v)}/K) \oplus
\Gal(K^{(w)}/K) ,
$$
which is  $(\Z/2)^3\oplus U_2^2$ as a $\F_2[\Gal(K/\mathbb{Q})]$-module.
Now, $\Gal(K_f/K)$ is a quotient of this direct sum whose Jordan-H\"{o}lder sequence contains a subset of the irreducible $\F_2[\Gal(K/\mathbb{Q})]$-modules occurring in the decomposition of $\Gal(M/K)$. 
Furthermore, since the order of $\Gal(K/\mathbb{Q})$ is prime to 2, $\Gal(K_f/K)$
 decomposes into a direct sum
$$
(\Z/2)^m\oplus U_2^n
$$
with $0 \leq m \leq 3$ and $0 \leq n \leq 2$.
Write
$$
f = f^{(0)} + f^{(1)} + f^{(2)}
$$
where
$$
f^{(0)} \in
H^1(G_S, 
\begin{pmatrix}
 & \\
 * & 
\end{pmatrix}
)
=
H^1(G_S, \F_4(\psi)), 
$$
$$
f^{(1)} \in
H^1(G_S, 
\begin{pmatrix}
 * & \\
 & *
\end{pmatrix}
)
=
H^1(G_S, \F_4^2), 
$$
and 
$$
f^{(2)} \in
H^1(G_S, 
\begin{pmatrix}
  & * \\
 & 
\end{pmatrix}
)
=
H^1(G_S, \F_4(\psi^{-1})).
$$
Suppose $f^{(0)} \neq 0.$ Then $$\text{Im}(f^{(0)}|_K) \simeq \F_4(\psi)$$ as  $\F_4[\Gal(K/\mathbb{Q})]$-modules.
On the other hand, 
$\text{Im}(f^{(0)}|_K)$ is a $\F_4[\Gal(K/\mathbb{Q})]$-submodule of $\text{Im}(f|_K)$, so we must have
$$\text{Im}(f^{(0)}) \simeq \Gal(K^{(w)}/K)$$
as $\F_4[\Gal(K/\mathbb{Q})]$-modules. 
This implies that $f_0 = c_0 f^{(w)}$ for some $c_0\in \F_4.$
By Lemma \ref{f^w}, $c_0=0$. 
Similar reasoning implies that $f^{(2)}= c_2 \cdot f^{(v)}$ for some 
$c_2\in \F_4.$
Since $f$ and $f^{(2)}$ are in the Selmer group, the homomorphism $f^{(1)}$ must also be in the Selmer group.
Since
$$\text{Im}(f^{(1)}|_K) \subseteq
\Gal(K(\sqrt{2})/K) \oplus
\Gal(K(\sqrt{v})/K) \oplus
\Gal(K(\sqrt{w})/K),$$
it follows that
$$
f^{(1)} = \sum_{q\in \{2,v,w \} }
\chi_q
\begin{pmatrix}
 a_q & \\
 & b_q
\end{pmatrix}
$$
for some $ a_q, b_q \in \F_4.$ It is easy to see that the only homomorphism of this form which lives in the Selmer group is a scalar multiple of $ \chi_2 I$.
\end{proof}

\section{Selmer ranks in families}
As above, let $S=\{\infty,2,7,l,v,w\}$.
Let $p$ be a trivial prime such that the the field $K^{(p)}$ above $K$ constructed in section \ref{irred} by means of the tame Gras-Munnier theorem has Galois group over $K$ isomorphic to
$\F_4(\psi^{-1})$ as $\F_4[\Gal(K/\mathbb{Q})]$-modules. Let $f^{(p)}$ be the corresponding cohomology class
in $H^1(G_{S\cup \{ p \}}, \F_4(\psi^{-1}))$. We will study the family of Selmer groups $H^1_\mathcal{N}(G_{S\cup \{ p \}}, Ad(\overline{\rho}))$ obtained by allowing ramification at such trivial primes $p$.
\begin{lemma}
Let $p$ be a trivial prime such that $\Gal(K^{(p)}/K)\simeq \F_4(\psi^{-1})$ as $\F_4[\Gal(K/\mathbb{Q})]$-modules.
Let $\mathcal{N}_p$ be the subspace of $H^1(G_p, Ad(\overline{\rho}))$ that preserves $\mathcal{C}_p^{\operatorname{unr},\geq 3}$. Then $\mathcal{N}_p$ is the span of the cohomology classes 
$$
f_p^{(1)}: 
\sigma_v \mapsto 
\begin{pmatrix}
0 & 1 \\
0 & 0
\end{pmatrix}
, \quad \tau_v \mapsto 
\begin{pmatrix}
0 & 0 \\
0 & 0
\end{pmatrix},
\quad 
f_p^{(2)}: 
\sigma_v \mapsto 
\begin{pmatrix}
0 & 0 \\
0 & 0
\end{pmatrix}
, \quad \tau_v \mapsto 
\begin{pmatrix}
0 & 1 \\
0 & 0
\end{pmatrix}
$$
and 
$$
f_p^{(3)}: 
\sigma_v \mapsto 
\begin{pmatrix}
1 & 0 \\
0 & 1
\end{pmatrix}
, \quad \tau_v \mapsto 
\begin{pmatrix}
0 & 0 \\
0 & 0
\end{pmatrix},
\quad 
g_p^{\text{\upshape{unr}}}: 
\sigma_v \mapsto 
\begin{pmatrix}
0 & 0 \\
1 & 0
\end{pmatrix}
, \quad \tau_v \mapsto 
\begin{pmatrix}
0 & 0 \\
0 & 0
\end{pmatrix}.
$$
We will also denote $g_p^{\text{\upshape{unr}}} $ by $f_p^{(4)}.$ 
\end{lemma}

\begin{proof}
Since $\rho_2$ is unramified at $p$, $\mathcal{N}_p = \mathcal{N}_p^{\operatorname{unr}}$. The statement follows by Corollary \ref{N_p}. 
\end{proof}

\begin{lemma}
The cohomology class $f^{(v)}$ belongs to $H^1_\mathcal{N}(G_{S\cup \{ p \}}, Ad(\overline{\rho})).$
\end{lemma}
\begin{proof}
The only thing left to check is the local condition at $p$. By Corollary \ref{N_p}, 
$$
f^{(v)}|_{G_p} \in H^1(G_p, 
\begin{pmatrix}
 & * \\
  & 
\end{pmatrix}) \subseteq \mathcal{N}_p.
$$
\end{proof}

\begin{lemma}
The cohomology class $\chi_2 I$ belongs to $H^1_\mathcal{N}(G_{S\cup \{ p \}}, Ad(\overline{\rho})).$
\end{lemma}
\begin{proof} Again, it suffices to consider the local condition at $p$:
$$
\chi_2 I |_{G_p} : \sigma_p \mapsto
\chi_2(\sigma_p)
\begin{pmatrix}
 1 & \\
  & 1 
\end{pmatrix},
\quad\tau_p \mapsto 0,
$$
which equals $\chi_2(\sigma_p) f_p^{(3)} \in \mathcal{N}_p$. (In fact $\chi_2(\sigma_p) = 1$ since $p \not \equiv 1 \mod 8$.)
\end{proof}

\begin{lemma} \label{f^p case1}
Suppose $f^{(w)}(\sigma_p) = 0.$ Then the cohomology class $f^{(p)}$ belongs to 
$H^1_\mathcal{N}(G_{S\cup \{ p \}}, Ad(\overline{\rho})).$
\end{lemma}
\begin{proof}
Since $f^{(w)}|_{G_p}= 0,$ it follows by the global reciprocity law, as in Lemma \ref{globalreciprocity}, that 
$f^{(p)}|_{G_w} = 0.$
It remains to check that $f^{(p)}|_{G_q} \in \mathcal{N}_q$ for the remaining primes $q$ in $S$. This is clear for 
$\infty, 7, l$ and $v$. By Lemma \ref{N_2}, 
$$
f^{(p)}|_{G_2} \in H^1(G_2, \begin{pmatrix}
 & * \\
 &
\end{pmatrix}) \subseteq \mathcal{N}_2.
$$
\end{proof}

\begin{lemma} Suppose $f^{(w)}(\sigma_p) \neq 0.$ Then the cohomology class $f^{(p)}$ does not belong to 
$H^1_\mathcal{N}(G_{S\cup \{ p \}}, Ad(\overline{\rho})).$
\end{lemma}
\begin{proof}
By Lemma \ref{globalreciprocity},
$$
f^{(p)}(\sigma_w) \neq 0.
$$
It follows that 
$$
f^{(p)}|_{G_w}: \sigma_w \mapsto 
\begin{pmatrix}
  & b\\
  & 
\end{pmatrix},
\quad
\tau_w \mapsto 0.
$$
for some $b\neq 0$. Let $f_w^{(1)},f_w^{(2)}, f_w^{(3)}$ and $g_w^\text{ram}$ be the four basis elements of $\mathcal{N}_w$ as in Lemmma \ref{N_w}, and let $f_w^{(4)} = g_w^\text{ram}$, which is defined as follows:
$$
g_w^{\text{ram}}: \sigma_w \mapsto 
\begin{pmatrix}
  & 1\\
  & 
\end{pmatrix}, \quad 
\tau_w \mapsto 
\begin{pmatrix}
 1 & \\
  & 1
\end{pmatrix}.
$$
Suppse
$$ f^{(p)}|_{G_w}
=\sum_{i=1}^4 \lambda_i f_w^{(i)}.
$$
Then 
$$
f^{(p)}(\sigma_w) = 
\begin{pmatrix}
  & b\\
  & 
\end{pmatrix} 
=
\begin{pmatrix}
  \lambda_3 & \lambda_4 \\
  \lambda_1 & \lambda_3  
\end{pmatrix}.
$$
Hence $\lambda_1,\lambda_3=0$ and $b=\lambda_4.$ Also,
$$
f^{(p)}(\tau_w) = 
\begin{pmatrix}
 0 & 0\\
 0 & 0
\end{pmatrix} 
=
\begin{pmatrix}
  \lambda_4 &  \\
  \lambda_2 & \lambda_4  
\end{pmatrix}.
$$
Hence $\lambda_4 =0$, and this contradiction shows that $f^{(p)}|_{G_w} \notin \mathcal{N}_w.$
\end{proof}

\begin{lemma} Let $T$ be a finite set of primes containing $S$ and let $M$ be a Galois stable subspace of $Ad(\overline{\rho})$ or $Ad(\overline{\rho})^*$ over $\F_4$.
The injective inflation map
$$
H^1(G_T , M) \to H^1(G_{ T \cup \{ p \}}, M)
$$
has cokernel of dimension equal to 
$$
\dim  H^2(G_p, M) = \dim M.
$$
\end{lemma}
\begin{proof}
This follows from Wiles' formula with $\mathcal{L}_q = H^1(G_q, M)$ for all $q\in T\cup \{ p\}$. Then 
$$H^1_\mathcal{L}(G_{T}, M) = H^1(G_{T}, M), 
\quad 
 H^1_\mathcal{L}(G_{T\cup \{ p\}}, M) = H^1(G_{T\cup \{ p\}}, M)
$$
and since $\mathcal{L}_q^\perp=0,$
$$H^1_\mathcal{L^\perp}(G_{T}, M^*) = \Sha^1_T(M^*),
\quad 
H^1_\mathcal{L^\perp}(G_{T\cup \{ p\}}, M^*) = \Sha^1_{T\cup \{ p\}}(M^*).
$$ 
By Lemma \ref{trivialsha}, 
$$\Sha^1_{T\cup \{ p\}}(M^*) \subseteq \Sha^1_T(M^*) =0.$$
It follows that 
\begin{IEEEeqnarray*}{ll}
 & \dim H^1 (G_T, M)  \\
= \quad & 
\dim H^0 (G_T, M) - \dim H^0 (G_T, M^*) 
+
\sum_{q\in T} \dim \mathcal{L}_q - \dim H^0(G_q, M),   
\end{IEEEeqnarray*}
while 
\begin{IEEEeqnarray*}{ll}
 & \dim H^1 (G_{T\cup \{ p\}}, M)   \\
= \quad & 
\dim H^0 (G_{T\cup \{ p\}}, M) - \dim H^0 (G_{T\cup \{ p\}}, M^*) 
+
\sum_{q\in T\cup \{ p\} } \dim \mathcal{L}_q - \dim H^0(G_q, M).   
\end{IEEEeqnarray*}
Clearly, the terms involving $H^0$ do not depend on the set of primes. Thus, 
\begin{IEEEeqnarray*}{rCl}
\dim H^1 (G_{T\cup \{ p\}}, M) &=& 
\dim H^1 (G_T, M) + \dim \mathcal{L}_p - \dim H^0(G_p, M) \\
&=&
\dim H^1 (G_T, M) + \dim H^1(G_p, M) - \dim H^0(G_q, M) \\
&=&
\dim H^1 (G_T, M) + \dim H^2(G_p, M).
\end{IEEEeqnarray*}
\end{proof}
The cohomology class $f^{(p)}$ is clearly a generator of the 1-dimensional cokernel of the inflation map
$$
H^1(G_S , 
\begin{pmatrix}
  & *\\
  & 
\end{pmatrix}
) \to H^1(G_{ S \cup \{ p \}}, 
\begin{pmatrix}
  & * \\
  & 
\end{pmatrix}
).
$$

\begin{lemma} The inflation map
$$
H^1(G_S , 
\begin{pmatrix}
 * & \\
 * & *
\end{pmatrix}
) \to H^1(G_{ S \cup \{ p \}}, 
\begin{pmatrix}
 * & \\
 * & *
\end{pmatrix}
)
$$
has a 3-dimensional cokernel. Let  $g$ be a cohomology class in 
$H^1(G_{ S \cup \{ p \}}, 
\begin{pmatrix}
 * & \\
 * & *
\end{pmatrix}
),$ and assume that $g$ is not in the image of the inflation map above.
Then $g$ is not in the Selmer group $H^1_\mathcal{N}(G_{S\cup \{ p \}}, Ad(\overline{\rho})).$
\end{lemma}
\begin{proof}
$$
g(\tau_p)
=
\begin{pmatrix}
 a & \\
 c & d
\end{pmatrix}
$$
where at least one of $a, c, d$ are not zero. By Corollary \ref{N_p}, it follows that $g|_{G_p} \notin \mathcal{N}_p$. 
\end{proof}

At this point, we have seen that the inflation of any element in $H^1_\mathcal{N}(G_S, Ad(\overline{\rho}))$ is also contained in $H^1_\mathcal{N}(G_{S\cup \{ p \}}, Ad(\overline{\rho}))$. It follows that
$$ \dim H^1_\mathcal{N}(G_{S\cup \{ p \}}, Ad(\overline{\rho}))
\geq 
\dim H^1_\mathcal{N}(G_S, Ad(\overline{\rho})).
$$

\begin{lemma}
Suppose $f^{(w)}(\sigma_p) = 0.$ Then 
$$ \dim H^1_\mathcal{N}(G_{S\cup \{ p \}}, Ad(\overline{\rho}))
=
\dim H^1_\mathcal{N}(G_{S}, Ad(\overline{\rho})) + 1,
$$
since
$$
H^1_\mathcal{N}(G_{S\cup \{ p \}}, Ad(\overline{\rho})) =\operatorname{span}_{\F_4} \{  
\chi_2 I, f^{(v)},  f^{(p)}
\}.
$$
\end{lemma}
\begin{proof} Since $f^{(w)}(\sigma_p) = 0,$ we know from Lemma \ref{f^p case1} that $f^{(p)} \in H^1_\mathcal{N}(G_{S\cup \{ p \}}, Ad(\overline{\rho}))$. We will show that $H^1_\mathcal{N}(G_{S\cup \{ p \}}, Ad(\overline{\rho}))$ equals the span of the three cohomology classes $f^{(p)}, f^{(v)}$ and $\chi_2 I$.
Suppose $ f\in H^1 (G_{S\cup \{ p \}}, Ad(\overline{\rho}))$, and write
$$
f = b f^{(p)} + g + f^{(S)},
$$
where $f^{(S)}$ is in the image of the inflation map
$$ 
 H^1(G_S, Ad(\overline{\rho}))
\to 
H^1(G_{S\cup \{ p \}}, Ad(\overline{\rho})),
$$
 while $ g \in H^1(G_{ S \cup \{ p \}}, 
\begin{pmatrix}
 * & \\
 * & *
\end{pmatrix}
)$ is not in the image of inflation. 
Then
$$
f(\tau_p) = b f^{(p)}(\tau_p) + g(\tau_p) 
=
\begin{pmatrix}
 a & b \\
 c & d
\end{pmatrix},
$$
where at least one of $a, c $ and $d$ is not zero. But then $f|_{G_p} \notin \mathcal{N}_p$ by Corollary \ref{N_p}.
Hence, if $ f\in H^1_\mathcal{N} (G_{S\cup \{ p \}}, Ad(\overline{\rho}))$, then 
$$
f = b f^{(p)} + f^{(S)},
$$
where $f^{(S)}$ is in the image of the inflation map as above. Since $f^{(p)} \in H^1_\mathcal{N} (G_{S\cup \{ p \}}, Ad(\overline{\rho})),$ it follows that also $f^{(S)} \in H^1_\mathcal{N} (G_{S\cup \{ p \}}, Ad(\overline{\rho}))$.
In particular, $f^{(S)} \in H^1_\mathcal{N} (G_{S}, Ad(\overline{\rho}))$, so 
$f^{(S)}$ is a linear combination of $f^{(v)}$ and $\chi_2 I$.
\end{proof}

\begin{lemma} \label{lemmaQ_2}
The quadratic character $\chi_w$ corresponding to $\Q(\sqrt{w})$ has the property that 
$$\chi_w
\begin{pmatrix}
a & \\
& b
\end{pmatrix}|_{G_2} \notin \mathcal{N}_2.
$$
for any $a,b \in \F_4$ not both equal to zero.
\end{lemma}
\begin{proof} Let $a,b \in \F_4$ not both equal to zero.
In Lemma \ref{chi_w}, we considered $\chi_w$ and showed that $\chi_w
\begin{pmatrix}
a & \\
& b
\end{pmatrix}|_{G_w} \notin \mathcal{N}_w.
$
Below, we will show that 
$\chi_w|_{G_2^{\text{ab}}}$ does not vanish on the torsion part of $G_2^{\text{ab}}$.
By Lemma \ref{N_2}, this implies that
$\chi_w
\begin{pmatrix}
a & \\
& b
\end{pmatrix}|_{G_2} \notin \mathcal{N}_2.
$ 

There are 7 nonisomorphic quadratic extensions of $\Q_2$. If $d$ is odd, the field $\Q_2(\sqrt{d})$ is determined by $d \mod 8$. For each $d$, let $H_d$ denote the subgroup of
$\Gal(\Q_2^{ab}/\Q_2)$ that fixes $\Q_2(\sqrt{d})$.
Since $w \equiv 3 \mod 4$, it holds that either $w\equiv -1 \mod 8$ or $w\equiv -5 \mod 8$.

First, suppose that $w\equiv -1 \mod 8$. Then $\Q_2(\sqrt{w})=\Q_2(\sqrt{-1}).$
Write
$$\Gal(\Q_2^{ab}/\Q_2)\simeq \hat{\Z}\times \Z_2 \times \Z/2\Z, $$
where $\Z_2 \times \Z/2\Z$ is the inertia subgroup.
By definition,
$$
H_{-1} = \hat{\Z} \times \Z_2 \times\{0 \},
$$
so the torsion part is not contained in $H_{-1}$. Therefore $\chi_{-1}|_{G_2}$ does not vanish on the torsion, so neither does $\chi_{w}|_{G_2}$. By Lemma \ref{N_2}, the desired result is now proved for $w\equiv -1 \mod 8.$

Second, suppose that $w\equiv -5 \mod 8$ so that $\Q_2(\sqrt{w})=\Q_2(\sqrt{-5}).$
Consider the diagram of fields below.
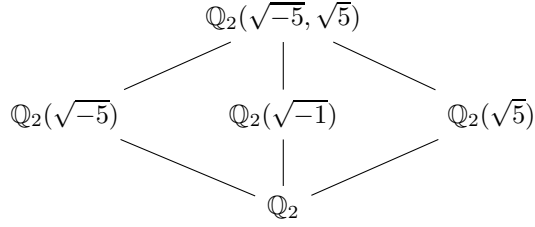
\[ 
\begin{tikzcd}
& \Q_2(\sqrt{-5},\sqrt{5}) \arrow[dl,dash] \arrow[d, dash] \arrow[dr,dash] &\\
\Q_2(\sqrt{-5}) \arrow[dr, dash]& \Q_2(\sqrt{-1}) \arrow[d, dash]& \Q_2(\sqrt{5})\arrow[dl, dash]\\
& \Q_2 & \\
\end{tikzcd}
\]
Note that $H_{-5} \cap H_5$ is the subgroup fixing $\Q_2(\sqrt{-5},\sqrt{5})$.
Since $\Q_2(\sqrt{-1})$ is the diagonal field of $\Q_2(\sqrt{-5})$ and $\Q_2(\sqrt{5})$, it follows that
$$
H_{-5} \cap H_5 \subseteq H_{-1}.
$$
Evidently, $\Q_2(\sqrt{5})$ is the unique unramified quadratic extension of $\Q_2$. Indeed, 
$\Q_2(\sqrt{5})=\Q_2(\sqrt{-3})$, which is obtained by adjoining a third root of unity to $\Q_2$. The minimal polynomial is therefore $x^2+x+1,$ and this polynomial is irreducible over $\F_2,$ so the associated residue field extension is of degree 2 over $\F_2$.
We conclude that 
$$
H_5 = 2\hat{\Z}\times \Z_2\times \Z/2\Z.
$$
Suppose $H_{-5}$ contains the torsion subgroup $\{0\}\times \{0\} \times \Z/2\Z.$ Then 
$$
\{0\}\times \{0\} \times \Z/2\Z \subseteq H_{-5} \cap H_{5} \subseteq H_{-1}.
$$
This is a contradiction. It follows that $\chi_{-5}|_{G_2}$ does not vanish on the torsion subgroup, so neither does $\chi_{w}|_{G_2}$. By Lemma \ref{N_2}, the proof is now finished $w\equiv -5 \mod 8$ as well.
\end{proof}

\begin{lemma}
Suppose $f^{(w)}(\sigma_p) \neq 0.$ Then 
$$ \dim H^1_\mathcal{N}(G_{S\cup \{ p \}}, Ad(\overline{\rho}))
=
\dim H^1_\mathcal{N}(G_{S}, Ad(\overline{\rho})),
$$
and
$$
H^1_\mathcal{N}(G_{S\cup \{ p \}}, Ad(\overline{\rho})) =\text{\upshape{span}}_{\F_4} \{  
\chi_2 I, f^{(v)}
\}.
$$
\end{lemma}
\begin{proof} Assume that $f^{(w)}(\sigma_p) \neq 0.$ 
By Lemma \ref{globalreciprocity}, it follows that 
$f^{(p)}(\sigma_w) \neq 0.$
We may normalize $f^{(p)}$ so that 
$$
f^{(p)}(\sigma_w) =
\begin{pmatrix}
 & 1 \\
 &
\end{pmatrix}.
$$
We have already seen that 
$$
H^1_\mathcal{N}(G_{S\cup \{ p \}}, Ad(\overline{\rho})) \supseteq \operatorname{span}_{\F_4} \{  
\chi_2 I, f^{(v)} 
\}.
$$
Let $h \in H^1_\mathcal{N}(G_{S\cup \{ p \}}, Ad(\overline{\rho})).$ We may write
$$
h = b f^{(p)} + g + f^{(S)},
$$
where $f^{(S)}$ inflates from a cohomology class in $H^1(G_{S}, Ad(\overline{\rho}))$, and 
where $g\in H^1(
G_{S\cup \{ p \}}, 
\begin{pmatrix}
 * & \\
 * & * 
\end{pmatrix}
)$. If $g$ does not inflate from $H^1(G_{S}, Ad(\overline{\rho}))$, then 
$$
h(\tau_p) = 
\begin{pmatrix}
 a & b\\
 c & d 
\end{pmatrix}
$$
where $a, c$ or $d$ is nonzero. But then $h|_{G_p} \notin \mathcal{N}_p$ by Corollary \ref{N_p}. Hence we may assume that $g=0$ and 
$$
h = bf^{(p)} + f^{(S)}.
$$
By assumption, $h|_{G_q} \in \mathcal{N}_q$ for all $q\in S\cup\{p \}$, and it is clear that
$f^{(p)}|_{G_q} \in \mathcal{N}_q$ for all $q\in (S \cup\{p\}) \setminus \{ w \}$. It follows that
$f^{(S)}|_{G_q} \in \mathcal{N}_q$ for all $q\in (S \cup\{p\}) \setminus \{ w \}$.
In particular, $f^{(S)} \in H^1(G_S, Ad(\overline{\rho}))$ is unramified at $\infty, 7,$ and $l$, so  
$\text{Im}(f^{(S)}|_K)$ is a $\Gal(K/\Q)$-submodule of the maximal 2-elementary abelian extension of $K$ unramified outside $2, v$ and $w.$ It follows that 
$$
f^{(S)} = 
 \sum_{q \in \{2,v,w \} }
 \chi_q
\begin{pmatrix}
 a_q & \\
 & b_q
\end{pmatrix}
+ c_v f^{(v)} + c_w f^{(w)}
$$
for some $a_q, b_q, c_v,c_w  \in \F_4.$
We conclude that 
$$ H^1_\mathcal{N}(G_{S\cup \{ p \}}, Ad(\overline{\rho})) \ni 
h -  c_v f^{(v)}  = bf^{(p)} +  c_w f^{(w)} +\sum_{q \in \{2,v,w \} }
\chi_q
\begin{pmatrix}
 a_q & \\
 & b_q
\end{pmatrix}.
$$
Next, we will show that $a_v =b_v =0.$ Indeed, 
$(b f^{(p)}+c_w f^{(w)})|_{G_v} \in \mathcal{N}_v,$ so 
$$ \sum_{q \in \{2,v,w \} }
\chi_q
\begin{pmatrix}
 a_q & \\
 & b_q
\end{pmatrix} |_{G_v} \in \mathcal{N}_v.
$$
As in Lemma \ref{N_v}, let $\mathcal{N}_v = \operatorname{span} \{ f_v^{(i)}: i=1,\ldots,4 \}$, where
$f_v^{(4)}=g_v^{\operatorname{unr}}$. There are coefficients $\lambda_v^{(i)}$ such that 
$$ \sum_{q \in \{2,v,w \} }
\chi_q
\begin{pmatrix}
 a_q & \\
 & b_q
\end{pmatrix} (\sigma_v)
= \sum_{i=1}^4 \lambda_v^{(i)} f_v^{(i)}(\sigma_v) 
=
\begin{pmatrix}
\lambda_v^{(3)} & \lambda_v^{(1)} \\
\lambda_v^{(4)} & \lambda_v^{(3)}
\end{pmatrix}.
$$
It follows that $\lambda_v^{(4)}=0.$ Also,
$$ \sum_{q \in \{2,v,w \} }
\chi_q
\begin{pmatrix}
 a_q & \\
 & b_q
\end{pmatrix} (\tau_v)
=
\begin{pmatrix}
a_v & \\
& b_v
\end{pmatrix}
= \sum_{i=1}^4 \lambda_v^{(i)} f_v^{(i)}(\tau_v) 
=
\begin{pmatrix}
\lambda_v^{(4)}  & \lambda_v^{(2)} \\
 & \lambda_v^{(4)} 
\end{pmatrix}.
$$
Hence $a_v = b_v = 0.$ 
We conclude that 
$$ H^1_\mathcal{N}(G_{S\cup \{ p \}}, Ad(\overline{\rho})) \ni 
h -  c_v f^{(v)}  = b f^{(p)} + c_w f^{(w)} +  \sum_{q \in \{2,w \} }
\chi_q
\begin{pmatrix}
 a_q & \\
 & b_q
\end{pmatrix}.
$$
As in Lemma \ref{N_w}, let $\mathcal{N}_w = \operatorname{span} \{ f_w^{(i)}: i=1,\ldots,4 \}$, where $f^{(4)}_w = g_w^{\operatorname{ram}}$.
There are coefficients $\lambda_w^{(i)}$ such that 
$$
b f^{(p)}|_{G_w} +  \sum_{q \in \{2,w \} }
\chi_q
\begin{pmatrix}
 a_q & \\
 & b_q
\end{pmatrix}|_{G_w}
=\sum_{i=1}^4 \lambda_w^{(i)} f_w^{(i)}.
$$
Since $w \equiv 3 \mod 4$, $\chi_2(\sigma_w)=1$, so
$$
bf^{(p)}(\sigma_w) +  \sum_{q \in \{2,w \} }
\chi_q
\begin{pmatrix}
 a_q & \\
 & b_q
\end{pmatrix} (\sigma_w)
=
\begin{pmatrix}
  a_2 & b \\
  & b_2
\end{pmatrix} 
=
\begin{pmatrix}
  \lambda_w^{(3)} & \lambda_w^{(4)} \\
 \lambda_w^{(1)} & \lambda_w^{(3)}
\end{pmatrix} 
$$
We conclude that $a_2 = b_2$ and $b= \lambda_w^{(4)}$. Then 
$$
b f^{(p)}(\tau_w) +  \sum_{q \in \{2,w \} }
\chi_q
\begin{pmatrix}
 a_q & \\
 & b_q
\end{pmatrix} (\tau_w)
=
\begin{pmatrix}
  a_w &  \\
  & b_w
\end{pmatrix} 
=
\begin{pmatrix}
  \lambda_w^{(4)}  &  \\
 \lambda_w^{(2)} & \lambda_w^{(4)} 
\end{pmatrix}. 
$$
This shows that $a_w = b_w = \lambda_w^{(4)}  = b $. 
In total,
$$ H^1_\mathcal{N}(G_{S\cup \{ p \}}, Ad(\overline{\rho})) \ni 
h -  c_v f^{(v)}  - a_2\chi_2 I = b f^{(p)} + c_w f^{(w)} + b \chi_w I.
$$
Since $f^{(p)}|_{G_2}  \in \mathcal{N}_2$, then also 
$(c_w f^{(w)} + b \chi_w I)|_{G_2}  \in \mathcal{N}_2.$ 
Since $$
c_w f^{(w)}|_{G_2} \in 
H^1(G_2, 
\begin{pmatrix}
    & \\
  *  &
\end{pmatrix}),
$$
it is a consequence of 
Lemma \ref{N_2} that $c_w=0.$
By Lemma \ref{lemmaQ_2}, 
$\chi_w I|_{G_2}  \notin \mathcal{N}_2.$ We conclude that $b=0.$ This concludes the proof.
\end{proof}

\begin{theorem}
There exist infinitely many trivial primes $p$ such that when we allow ramification at $p$,
the rank of the Selmer group increases by one, i.e.
$$ \dim H^1_\mathcal{N}(G_{S\cup \{ p \}}, Ad(\overline{\rho}))
=
\dim H^1_\mathcal{N}(G_{S}, Ad(\overline{\rho})) + 1.
$$
More precisely,
$$ \dim H^1_\mathcal{N}(G_{S\cup \{ p \}}, Ad(\overline{\rho}))
=
\begin{cases}
\dim H^1_\mathcal{N}(G_{S}, Ad(\overline{\rho})) + 1, & \text{if } f^{(w)}(\sigma_p) =0, \\
\dim H^1_\mathcal{N}(G_{S}, Ad(\overline{\rho})), &  \text{if }f^{(w)}(\sigma_p) \neq 0. \\
\end{cases}
$$
\end{theorem}
\begin{proof} The changes in the Selmer rank depending on whether or not $p$ splits completely in $K^{(w)}$ follow from our work above. 
The primes $p$ arise from a finite list of compatible abelian Chebotarev conditions on degree one primes in $K$.  
\end{proof}

\section{Lowering the Selmer rank}
In the previous section, we showed how to raise the Selmer rank by allowing ramification at a new trivial prime $p$. In this section, we examine the possibility of lowering the Selmer rank by changing bases locally at trivial primes. 
Let $\omega \in \operatorname{GL}(2,\F_4)$. If $p$ is a trivial prime, then $\overline{\rho}|_{G_p}$ is trivial and therefore unaffected by a change of basis: 
$$
\omega \overline{\rho}|_{G_p} \omega^{-1}= \overline{\rho}|_{G_p}.
$$
In \cite{HR}, it is shown how it is possible to lower the Selmer rank by changing bases locally at trivial primes. 
Suppose we change basis locally at $p$ by $\omega \in GL(2,\W)$, and set 
$$
\mathcal{C}_p^{(\omega)} = \omega \mathcal{C}_p \omega^{-1}, \quad 
\mathcal{N}_p^{(\omega)} = \omega \mathcal{N}_p \omega^{-1}.
$$
\begin{lemma}
Let
$$
\omega = 
\begin{pmatrix}
& 1\\
1 &
\end{pmatrix},
$$
and suppose we change bases by $\omega$ at $v$.
Then 
$$\dim H^1_{\mathcal{N}^{(\omega)}}(G_{S}, Ad(\overline{\rho}))=1,$$
and $ H^1_{\mathcal{N}^{(\omega)}}(G_{S}, Ad(\overline{\rho})) = \text{\upshape{span}}_{\F_4} \{ \chi_2 I\}.$
\end{lemma}
\begin{proof}
By Lemma \ref{N_v} and Lemma \ref{N_w}, it follows that $f^{(v)}|_{G_v} \notin \mathcal{N}_v^{(\omega)}$. We will show that $\chi_2 I$ is the only element that remains in the Selmer group. 
Let $h\in H^1_\mathcal{N}(G_{S}, Ad(\overline{\rho}))$. 
By Lemma \ref{rayclass} and arguments similar to those in the proof of Proposition \ref{Selmer_S}
$$ 
h =  c_v f^{(v)} + c_w f^{(w)}  +  \sum_{q \in \{2,v,w \} }
\chi_q
\begin{pmatrix}
 a_q & \\
 & b_q
\end{pmatrix}.
$$
By Lemma \ref{N_2}, $c_w=0.$
As in Lemma \ref{N_v}, let  $\mathcal{N}_v^{(\omega)} = \text{span} \{ \omega f_v^{(i)} \omega^{-1} : i=1,\ldots,4 \}$. Let $\lambda_v^{(i)} \in \F_4$ be such that 
$$h|_{G_v}=\sum_{i=1}^4 \lambda_v^{(i)} \omega f_v^{(i)} \omega^{-1}.$$ 
Then
$$ 
h(\tau_v)
=
\begin{pmatrix}
 a_v & c_v \\
 & b_v
\end{pmatrix}
=
\begin{pmatrix}
\lambda_v^{(4)}  &  \\
\lambda_v^{(2)} & \lambda_v^{(4)} 
\end{pmatrix}.
$$
Hence $c_v=0,$ and $a_v = b_v = \lambda_v^{(4)}.$ 
Consequently, 
$$ 
h(\sigma_v)
= 
\begin{pmatrix}
 a_2 &  \\
 & b_2
\end{pmatrix} 
+\chi_w(\sigma_v)
\begin{pmatrix}
 a_w &  \\
 & b_w
\end{pmatrix}
=
\begin{pmatrix}
\lambda_v^{(3)}  & \lambda_v^{(4)}  \\
\lambda_v^{(1)} & \lambda_v^{(3)}
\end{pmatrix}.
$$
Thus $\lambda_v^{(4)}=0,$ so $a_v = b_v =0$, and 
$$ 
h =    \sum_{q \in \{2,w \} }
\chi_q
\begin{pmatrix}
 a_q & \\
 & b_q
\end{pmatrix}.
$$
As in Lemma \ref{N_w},  let $\mathcal{N}_w^{(\omega)} = \text{span} \{ \omega f_w^{(i)} \omega^{-1} : i=1,\ldots,4 \}$ and let $\lambda_w^{(i)} \in \F_4$ be such that 
$$h|_{G_w}=\sum_{i=1}^4 \lambda_w^{(i)} \omega f_w^{(i)} \omega^{-1}.$$ 
Then
$$ 
h(\tau_w)
=
\begin{pmatrix}
 a_w &  \\
  & b_w
\end{pmatrix}
=
\begin{pmatrix}
\lambda_w^{(4)} & \lambda_w^{(2)} \\
 & \lambda_w^{(4)}
\end{pmatrix}.
$$
Hence $a_w = b_w = \lambda_w^{(4)} .$ It follows that
$$ 
h(\sigma_w)
= 
\begin{pmatrix}
 a_2 &  \\
 & b_2
\end{pmatrix} 
=
\begin{pmatrix}
\lambda_w^{(3)}  & \lambda_w^{(1)}  \\
\lambda_w^{(4)} & \lambda_w^{(3)}
\end{pmatrix}.
$$
Hence $\lambda_w^{(4)}=0$, so $a_w = b_w = \lambda_w^{(4)} = 0,$ while $a_2 = b_2 =\lambda_w^{(3)}.$ 
We conclude that
$$
h = a_2 \chi_2 I,
$$
and this concludes the proof.
\end{proof}

As a result of working with the full adjoint, our space $\mathcal{N}_p$ is one dimension larger than in \cite{HR}; the extra dimension comes from 
the cohomology class 
$$
f_p^{(3)}: \sigma_p \mapsto I, \tau_p \mapsto 0,
$$
which is contained in $\mathcal{N}_p$ regardless of whether $p$ ramifies in $\rho_2$ or not.
The class $f_p^{(3)}$ will always remain in $\mathcal{N}_p$ after any change of basis at $p$. Indeed, for any $\omega \in GL(2, \F_4),$ and any trivial prime $p$,
$$
\chi_2 I|_{G_p} = \chi_2(\sigma_p) f_p^{(3)} = \chi_2(\sigma_p) \omega f_p^{(3)} \omega^{-1} \in \mathcal{N}_p^{(\omega)}.
$$
So the cohomology class $\chi_2 I$ will always remain in the Selmer group after changing basis by $\omega$ at the trivial primes, and this keeps the Selmer rank at least one. 

\section{Density of Selmer ranks in families}
\begin{theorem} \label{density1}
Let $K(\sqrt{U_K})$ be the governing field of $K$ with 
$$
\Gal(K(\sqrt{U_K})/K) = \F_2 \oplus U_2
$$
as $\F_2[\Gal(K/\Q)]$-modules. 
The density of primes $p$ such that $p$ is trivial and such that 
$\sigma_{\mathfrak{p}} =(1,0) \in \F_2 \oplus U_2$ for some $\mathfrak{p}$ dividng $p$ in $K$ is equal to
1/24. 
\end{theorem}
\begin{proof}
We apply Chebotarev's theorem to the extension $K(\sqrt{U_K})/\Q$. Let $\varphi$ denote the action of 
$\Z/3\Z$ on $\Gal(K(\sqrt{U_K})/K)= \F_2 \oplus U_2.$
The group
$\Gal(K(\sqrt{U_K})/\Q)$ is isomorphic to the semidirect product
$$
0 \to \F_2 \oplus U_2 \to \Big( \F_2 \oplus U_2 \Big) \rtimes_\varphi \Z/3\Z \stackrel{\pi}{\to} \Z/3\Z \to 0.
$$
The element $(1,0)\rtimes_\varphi 0$ lies in the centre of $\Gal(K(\sqrt{U_K})/\Q)$. Indeed, if 
$(a,b)\rtimes_\varphi n$ is an arbitrary element in $\F_2 \oplus U_2 \rtimes_\varphi \Z/3\Z$,
we find that
\begin{IEEEeqnarray*}{rCl}
\Big( (a,b)\rtimes_\varphi n \Big) \cdot \Big( (1,0)\rtimes_\varphi 0 \Big) 
&=& ( (a,b) + \varphi(n) (1,0)  ) \rtimes_\varphi (n+0) \\
&=& ((a,b)+(1,0))\rtimes_\varphi n
\end{IEEEeqnarray*}
since $\Z/3\Z$ acts trivially on $(1,0) \in \F_2 \oplus U_2.$ On the other hand,
\begin{IEEEeqnarray*}{rCl}
 \Big( (1,0)\rtimes_\varphi 0 \Big) \cdot 
\Big( (a,b)\rtimes_\varphi n \Big)
&=& ( (1,0) + \varphi(0) (a,b)  ) \rtimes_\varphi (0+n) \\
&=& ((1,0)+(a,b))\rtimes_\varphi n.
\end{IEEEeqnarray*}
It follows that if $C$ denotes the conjugacy class in $\Gal(K(\sqrt{U_K})/\Q)$ containing the element 
$(1,0)\rtimes_\varphi 0$, then $$C=\{ (1,0)\rtimes_\varphi 0\}.$$
The set of primes $p$ in $\Q$ such that the Frobenius conjugacy class $(p, K(\sqrt{U_K})/\Q)$ is equal to $C$ is the same as the set of primes $p$ such that 
$$\sigma_{\overline{p}/p} = (1,0)\rtimes_\varphi 0$$
for any prime $\overline{p}$ in $K(\sqrt{U_K})$ lying over $p$. By Chebotarev's theorem, the natural density of such primes is equal to $1/24$. 
Note that $\pi (p, K(\sqrt{U_K})/\Q) $ is equal to the Frobenius element $(p, K/\Q)$.
Hence, a prime $p$ splits completely in $K$ if and only if $\pi  (p, K(\sqrt{U_K})/\Q) = 0$.
Let $\mathfrak{p}$ be a prime in $K$ below $\overline{p}$. If $p$ is a prime that splits completely in $K$, then 
$$
\sigma_{\overline{p}/p} = (1,0)\rtimes_\varphi 0
$$
if and only if $\sigma_{\overline{p}/\mathfrak{p}} = (1,0)$ in $\F_2\oplus U_2.$
The theorem follows.
\end{proof}

\begin{theorem}For a prime $p$, let $\mathfrak{p}$ be a prime in $K$ above $p$, and let 
$\sigma_\mathfrak{p}$ denote the Froenius element $(\mathfrak{p}, K(\sqrt{U_K})/K)$. 
The density of primes $p$ such that 
\begin{itemize}
\item $p$ splits completely in $K^{(w)}$ 
\item The Frobenius element $\sigma_\mathfrak{p} = (1,0) $ for any $\mathfrak{p}$ dividing $p$ in $K$ 
\item p remains prime in $\Q(\sqrt{l})$ 
\end{itemize}
is equal to 1/192. The first two conditions imply in particular that $p$ is trivial. 
\end{theorem}

\begin{proof}
The proof is similar to the proof of Theorem \ref{density1}. We apply Chebotarev's theorem to the extension $K(\sqrt{U_K}) K^{(w)} K(\sqrt{l})/\Q$. 
Let $U_2^{(w)}$ denote $\Gal(K^{(w)}/K)$ as a $\Gal(K/\Q)$-module.
The Galois group
$\Gal(K(\sqrt{U_K}) K^{(w)} K(\sqrt{l})/\Q)$ is isomorphic to the semidirect product
$$
0 \to \F_2 \oplus U_2 \oplus U_2^{(w)} \oplus \F_2 \to \F_2 \oplus U_2 \oplus U_2^{(w)} \oplus \F_2 \rtimes_\varphi \Z/3\Z \stackrel{\pi}{\to} \Z/3\Z \to 0.
$$
The element $(1,0,0,1) \rtimes_\varphi 0$ lies in the center.
A prime $p$ splits completely in $K^{(w)}$, has $(\mathfrak{p}, K(\sqrt{U_K})/K) = (1,0) $, and remains prime in $\Q(\sqrt{l})$ if and only if 
$$
(p, K(\sqrt{U_K}) K^{(w)} K(\sqrt{l})/\Q) = \{(1,0,0,1) \rtimes_\varphi 0 \}.
$$
By Chebotarev's theorem, the density of such primes is 1/192.
\end{proof}

\begin{corollary} \label{1/192}
The density of primes $p$ such that  $\rho_2|_{G_p} \in \mathcal{C}_p$ and 
$$ \dim H^1_\mathcal{N}(G_{S\cup \{ p \}}, Ad(\overline{\rho}))
=
\dim H^1_\mathcal{N}(G_{S}, Ad(\overline{\rho})) + 1
$$
is 1/192.
\end{corollary}

\begin{corollary}
The probability that a prime $p$ is trivial and has
$\sigma_{\mathfrak{p}} =(1,0) \in \F_2 \oplus U_2$ for some $\mathfrak{p}$ dividing $p$ in $K$ is equal to 4.166\%. Among such primes $p$, it holds that $\rho_2|_{G_p} \in \mathcal{C}_p$ and 
$$ \dim H^1_\mathcal{N}(G_{S\cup \{ p \}}, Ad(\overline{\rho}))
=
\dim H^1_\mathcal{N}(G_{S}, Ad(\overline{\rho})) + 1
$$
an eighth of the time. 
\end{corollary}

\begin{corollary}
    The global, even deformation ring 
    $R_{(\mathcal{N}_q)_{q\in S}}$
    parametrizing deformations $\rho$ of 
    $\overline{\rho}$ such that $\rho|_{G_q} \in \mathcal{C}_q$ for all $q\in S$
   has the structure 
$$
R_{(\mathcal{N}_q)_{q\in S}} \simeq \W[[T_1,T_2]]/(r_1,r_2, r_3).
$$
There exists an infinite set of primes $p$ of density 1/192 such that
$R_{(\mathcal{N}_q)_{q\in S\cup\{p\}}}$ has the structure 
$$
R_{(\mathcal{N}_q)_{q\in S\cup\{p\}}}
 \simeq \W[[T_1,T_2, T_3]]/(r_1,r_2, r_3, r_4).
$$
\end{corollary}
\begin{proof}
   This follows by combining 
Proposition \ref{Selmer_S}, \ref{globalsetting}, and \ref{1/192}
   with
   Proposition \ref{R_N}. In addition, we use that 
   $$
\dim H^1_{\mathcal{N}^\perp}(G_{S\cup \{ p\} }, Ad(\overline{\rho})^*)
=
\dim H^1_\mathcal{N}(G_{S\cup \{p \}}, Ad(\overline{\rho})) + 1,
$$
which follows from Wiles' formula (Proposition \ref{Wiles}). 
\end{proof}

\bibliographystyle{amsplain} 
\bibliography{references.bib}

\providecommand{\bysame}{\leavevmode\hbox to3em{\hrulefill}\thinspace}
\providecommand{\MR}{\relax\ifhmode\unskip\space\fi MR }
\providecommand{\MRhref}[2]{%
  \href{http://www.ams.org/mathscinet-getitem?mr=#1}{#2}
}
\providecommand{\href}[2]{#2}
\begin{thebibliography}{10}

\bibitem{rep}
\emph{\href{https://groupprops.subwiki.org/wiki/Linear_representation_theory_of_cyclic_group:Z3}{Linear
  representation theory of cyclic group:Z3}}, {Accessed: 2022-05-21}.

\bibitem{Calegari}
F.~Calegari, \emph{Reciprocity in the {L}anglands program since {F}ermat's
  {L}ast {T}heorem},  (2022),
  \href{https://arxiv.org/abs/2109.14145}{arXiv:2109.14145}.

\bibitem{Christine2}
S.~Chan, C.~McMeekin, and D.~Milovic,
  \emph{\href{https://doi.org/10.1007/s40993-021-00295-5}{A density of ramified
  primes}}, Research in {N}umber {T}heory \textbf{8} (2022), no.~1.

\bibitem{FKP}
N.~Fakhruddin, C.~Khare, and S.~Patrikis, \emph{Lifting and automorphy of
  reducible mod $p$ {G}alois representations over global fields},  (2021),
  \href{https://arxiv.org/abs/2008.12593}{arXiv:2008.12593v5}.

\bibitem{FKP1}
\bysame, \emph{Relative deformation theory, relative selmer groups, and lifting
  irreducible {G}alois representations}, Duke Mathematical Journal \textbf{170}
  (2021), no.~16, 3505–--3599.

\bibitem{FIMR}
J.~B. Friedlander, H.~Iwaniec, B.~Mazur, and K.~Rubin, \emph{{The spin of prime
  ideals}}, Inventiones mathematicae \textbf{193} (2013), 697–--749.

\bibitem{gras}
F.~Hajir, C.~Maire, and R.~Ramakrishna, \emph{On tame
  $\mathbb{Z}/p\mathbb{Z}$-extensions with prescribed ramification}, 2022,
  \href{https://arxiv.org/abs/2208.05007}{arXiv:2208.05007}.

\bibitem{HR}
S.~Hamblen and R.~Ramakrishna, \emph{{Deformations of certain reducible Galois
  representations, II}}, American Journal of Mathematics \textbf{130} (2008),
  no.~4, 913--944.

\bibitem{KLR2}
C.~Khare, M.~Larsen, and R.~Ramkrishna, \emph{{Transcendental l-adic Galois
  representations}}, Mathematical Research Letters \textbf{12} (2004), no.~5.

\bibitem{KLR1}
\bysame, \emph{{Constructing semisimple p-adic Galois representations with
  prescribed properties}}, American Journal of Mathematics \textbf{127} (2005),
  no.~4, 709--734.

\bibitem{KR}
C.~Khare and R.~Ramakrishna, \emph{{Finiteness of Selmer groups and deformation
  rings}}, Inventiones Mathematicae \textbf{154} (2003), 179–198.

\bibitem{KW1}
C.~Khare and J.-P. Wintenberger, \emph{{Serre's modularity conjecture (I)}},
  Inventiones Mathematicae \textbf{178} (2009), no.~3, 485–--504.

\bibitem{KW2}
\bysame, \emph{{Serre's modularity conjecture (II)}}, Inventiones Mathematicae
  \textbf{178} (2009), no.~3, 505–586.

\bibitem{koch}
H.~Koch, \emph{{Galois theory of $p$-extentions}}, Springer, 2002.

\bibitem{Koymans}
P.~Koymans and D.~Milovic, \emph{{Joint distribution of spins}}, Duke Math. J.
  \textbf{170} (2021), no.~8, 1723–1755.

\bibitem{Lang}
S.~Lang, \emph{{Algebra}}, revised third edition ed., Springer, 2002.

\bibitem{Mazur1}
B.~Mazur, \emph{Deforming {G}alois representations}, Galois groups over
  $\mathbb{Q}$ (Y.~Ihara, K.~Ribet, and J.-P. Serre, eds.), Springer, New York,
  1989, pp.~385--437.

\bibitem{Mazur2}
\bysame, \emph{An introduction to the deforming theory of {G}alois
  representations}, Modular forms and {F}ermat's last theorem (G.~Cornell,
  J.~H. Silverman, and G.~Stevens, eds.), Springer, New York, 1997.

\bibitem{Christine1}
C.~McMeekin, \emph{{On the asymptotics of a prime spin relation}}, Journal of
  Number Theory \textbf{200} (2019), 407--426.

\bibitem{Patrikis1}
Stefan Patrikis, \emph{{{D}eformations of {G}alois representations and
  exceptional monodromy}}, Inventiones mathematicae \textbf{205} (2016),
  269–336.

\bibitem{Paskunas}
V.~Pa\v{s}k\={u}nas, \emph{{On 2-adic deformations}}, 2022,
  \href{https://arxiv.org/abs/1509.00320}{arxiv.org/abs/1509.00320}.

\bibitem{even1}
R.~Ramakrishna, \emph{{Deforming an even representation}}, Inventiones
  mathematicae \textbf{322} (1998), no.~3, 563--580.

\bibitem{even2}
\bysame, \emph{{Deforming an even representation II, Raising the level}},
  Journal of Number Theory \textbf{72} (1998), no.~1, 92--109.

\bibitem{lifting}
\bysame, \emph{{Lifting Galois representations}}, Inventiones mathematicae
  \textbf{138} (1999), 537–--562.

\bibitem{SFM}
\bysame, \emph{{Deforming Galois representations and the conjectures of Serre
  and Fontaine-Mazur}}, Annals of Mathematics \textbf{156} (2002), no.~1,
  115--154.

\bibitem{canadian}
\bysame, \emph{{Constructing Galois Representations with Very Large Image}},
  Canadian Journal of Mathematics \textbf{60} (2008), no.~1, 208--221.

\bibitem{Ribet}
K.~Ribet, \emph{On modular representations of $\operatorname{Gal}(\overline{
  \mathbb{Q} }/\mathbb{Q})$ arising from modular forms}, Inventiones Math.
  \textbf{100} (1990), 431--476.

\bibitem{artinII}
R.~Taylor, \emph{{On icosahedral Artin representations, II}}, American Journal
  of Mathematics \textbf{125} (2002), 549--566.

\bibitem{Taylor-Wiles}
R.~Taylor and A.~Wiles, \emph{{Ring-theoretic properties of certain Hecke
  algebras}}, Annals of Mathematics \textbf{141} (1995), no.~3, 553--572.

\bibitem{Wiles}
A.~Wiles, \emph{{Modular elliptic curves and Fermat's last theorem}}, Annals of
  Mathematics \textbf{141} (1995), no.~3, 443--551.

\end{thebibliography}

\nocite{*}

\end{document}